\documentclass[11pt,a4paper,reqno]{amsart}
\usepackage{xcolor}
\usepackage{geometry}
\usepackage{mathtools}
\usepackage{amsfonts}
\usepackage{amsmath, amssymb, amsthm}
\usepackage{hyperref}
\usepackage{graphicx}
\usepackage{bbm}
\usepackage{cite}
\usepackage{subfigure}
\usepackage{tikz}
\usetikzlibrary{patterns, arrows.meta}

\usepackage{epstopdf}
\usepackage{color}
\usepackage{amssymb}

\usepackage{amsthm}

\usepackage{mathrsfs}
\usepackage{dutchcal}
\usepackage{tikz}

\newcommand{\N}{\mathbb{N}}

\numberwithin{equation}{section}
\newtheorem{theorem}{Theorem}[section]
\newtheorem{lemma}[theorem]{Lemma}
\newtheorem{remark}[theorem]{Remark}

\newtheorem{proposition}[theorem]{Proposition}
\newtheorem{definition}[theorem]{Definition}
\newtheorem{example}[theorem]{Example}
\def\eqref#1{(\ref{#1})}
\allowdisplaybreaks
\def\enddoc{\end{document}}

\begin{document}
\author{Lu Hao}
\address{Universit\"{a}t Bielefeld, Fakult\"{a}t f\"{u}r Mathematik, Postfach 100131, D-33501, Bielefeld, Germany}
\email{lhao@math.uni-bielefeld.de}
	
\title[Wiener criterion at infinity on weighted graphs]{A Wiener criterion at infinity for $p$-massiveness on weighted graphs}

\thanks{\noindent L. Hao was funded by the Deutsche Forschungsgemeinschaft(DFG, German Research Foundation)-Project-ID317210226-SFB 1283.}


\begin{abstract}
We study boundary value problems at infinity for the graph $p$-Laplacian on infinite, connected, locally finite weighted graphs. Our main result is a Wiener criterion for $p$-massiveness. Assuming volume doubling and a weak $(1,p)$-Poincar\'e inequality, we show that every infinite connected $p$-massive set satisfies a dyadic capacitary condition expressed through relative $p$-capacities in nested balls; under the additional $(p_0)$ condition, the converse also holds. This yields a nonlinear criterion at the point at infinity in a rough weighted-graph setting and extends the Wiener viewpoint to a nonlinear discrete framework. We also prove, without these geometric assumptions, that $p$-massiveness is equivalent to a strengthened nonuniqueness property for exterior Dirichlet problems. As a further consequence, bounded nonconstant $p$-harmonic functions are characterized by the existence of two disjoint massive sets. In this way, the Wiener criterion is placed in a broader and more flexible picture of exterior boundary behavior and Liouville-type phenomena on weighted graphs.
\end{abstract}

\maketitle

\section{Introduction}

Boundary behavior at infinity is a classical theme in potential theory. In the linear case it is closely tied to random walks, escape phenomena, and Wiener-type tests, while in the nonlinear setting it is related to capacitary criteria, massiveness, and exterior boundary value problems. On weighted graphs, these questions lie at the intersection of nonlinear potential theory, discrete analysis, and the geometry of the underlying space.

In this paper we study the exterior Dirichlet problem and the $p$-massiveness of subsets on weighted graphs. Our main result is a Wiener criterion at infinity for $p$-massive sets under the rough geometric assumptions \hyperref[VD]{$\text{(VD)}$} and \hyperref[PI]{$(P_p)$}. It is expressed directly in terms of the exterior geometry of the set through the relative $p$-capacities of a dyadic family of condensers. In addition, and without using \hyperref[VD]{$\text{(VD)}$} or \hyperref[PI]{$(P_p)$}, we identify $p$-massiveness with a strengthened nonuniqueness property for exterior Dirichlet problems and relate bounded nonconstant $p$-harmonic functions to the existence of two disjoint massive sets. Thus the Wiener criterion is connected with both exterior boundary value problems and Liouville-type phenomena for the graph $p$-Laplacian.

Let $G=(V,E)$ be an infinite, connected, locally finite graph. Here $V$ denotes the vertex set and $E$ denotes the edge set. If there exists an edge connecting $x$ and $y$, we write $x\sim y$.

Let $\mu: V\times V\rightarrow[0,\infty)$ be an edge weight, and denote it by $\mu_{xy}:=\mu(x,y)$. We assume that $\mu_{xy}>0$ if and only if $x\sim y$ and that $\mu_{xy}=\mu_{yx}$. The vertex weight is denoted by $m>0$.

For $p>1$, the $p$-Laplacian on the graph is defined by
$$
\Delta_p u(x) := \frac{1}{m(x)} \sum_{y \sim x} \mu_{xy} |\nabla_{xy} u|^{p-2} \nabla_{xy} u,
$$
where $\nabla_{xy} u = u(y) - u(x)$.

Let \( \ell(V) \) denote the set of all real functions on \(V\), \( \ell_0(V) \) the subset of functions with finite support, and \( \mathcal{B}(V) \) the set of bounded functions. We call a function \(u\in \ell(\Omega)\) $p$-subharmonic in \(\Omega\) if $\Delta_pu\ge 0$ for all $x\in \Omega$, and $p$-superharmonic if $\Delta_pu\le 0$. A function is $p$-harmonic in $\Omega$ if it is both $p$-subharmonic and $p$-superharmonic, that is, if $\Delta_pu=0$.

Since $m$ does not affect harmonicity, we mainly use the canonical vertex measure
$$\mu(x)=\sum_{y\sim x}\mu_{xy}.$$
With this choice of edge and vertex weights, we write the weighted graph simply as $(V,\mu)$.

For any two vertices $x$ and $y$, let $d(x,y)$
be the minimal number of edges among all possible paths connecting $x$ and $y$ on  graph $(V,\mu)$, then
$d(\cdot,\cdot)$ is a distance function on $V\times V$, and called the graph distance.
Fix some vertex $o\in V$, and for $r>0$, denote
$$B(o,r):=\{x\in V|\ d(o,x)\leq r\}.$$

Let $\Omega^c=V\setminus\Omega$. For a subset $\Omega \subset V$, we define the vertex boundary $\partial \Omega$ and the edge boundary $\partial_e \Omega$ by
$$
\partial \Omega = \{ x \in \Omega^c \mid \exists y \in \Omega \text{ such that } x \sim y \},
$$
and
$$
\partial_e \Omega = \{ xy \in E \mid x \in \Omega \text{ and } y \in \partial \Omega \},
$$
and we set $\overline{\Omega}:=\Omega\cup \partial \Omega$.

For an infinite subset $\Omega \subset V$ and a bounded boundary datum $f \in \mathcal{B}(\partial \Omega)$, we study the uniqueness of bounded solutions to the exterior Dirichlet problem
\begin{equation}\label{dp}
    \begin{cases}
\Delta_{p}u=0, & \text{in } \Omega,\\ 
u=f, & \text{on } \partial\Omega.
\end{cases}
\end{equation}

This problem is closely related to the notion of a massive set.
\begin{definition}
A subset \(\Omega \subset V\) is said to be \emph{$p$-massive} if there exists a function \(u \in \ell(V)\) such that
\begin{gather*}
    0 < u < 1 \quad \text{in } \Omega, \\
    \Delta_p u = 0 \quad \text{on } \Omega, \\
    u = 1 \quad \text{on }  \Omega^c.
\end{gather*}
If, in addition, \(u\) satisfies \(D_p(u) < \infty\), then \(\Omega\) is called \emph{$D_p$-massive}, where the $p$-Dirichlet energy of a function $f$ is defined by
\begin{equation}
    D_p(f) = \frac{1}{2} \sum_{x, y \in V} \mu_{xy} |f(x) - f(y)|^p.
\end{equation}
\end{definition}

For $p=2$, massiveness has a direct probabilistic interpretation: it describes whether a random walk starting in the set has a strictly positive probability of escaping to infinity without hitting the boundary. In the lattice case $\mathbb{Z}^d$, It\^o and McKean \cite{IM60} proved a Wiener criterion for this phenomenon; see also \cite{Lam63, McK61, BC15}. Classical Wiener-type criteria for the regularity of finite boundary points are well developed in the nonlinear setting; see, e.g., \cite{KM92, KM94, Bj09}. Our concern here is instead the point at infinity and its relation to massiveness.

To state the theorem, we introduce the $p$-capacity. For subsets $K \subset U$, it is defined by
$$
cap_p(K,U) := \inf \{ D_p(f) \mid f \in \ell_0(U), f=1 \text{ on } K \},
$$
where $D_p(f) = \frac{1}{2} \sum_{x,y} \mu_{xy} |f(x)-f(y)|^p$ denotes the $p$-Dirichlet energy; see Section \ref{pre} for details.

We also need several geometric assumptions.
\begin{definition}
We say that the weighted graph $(V,\mu)$ satisfies the \hyperref[p0]{$(p_{0})$} condition if there exists $p_0>0$ such that
\begin{equation}\label{p0}
	\frac{\mu_{xy}}{\mu(x)}\geq \frac{1}{p_0}, \quad\mbox{when $y\sim x$} \tag{$p_{0}$}.
\end{equation}
\end{definition}

\begin{definition}
We say that the weighted graph $(V,\mu)$ satisfies the volume doubling condition \hyperref[VD]{$\text{(VD)}$} if there exists a constant \(C_D>0\) such that
\begin{equation}\label{VD}
	\mu(B(x,2r))\le C_D \mu(B(x,r)), \tag{$\text{VD}$}
\end{equation}
for every ball \(B(x,r)\subset V\).
\end{definition}

\begin{definition}
We say that the weighted graph $(V,\mu)$ admits the weak $(1,p)$-Poincar\'e inequality \hyperref[PI]{$(P_p)$} if there exists a constant $C_p > 0$ such that for every ball $B=B(x,r)$ and all $f\in \ell(V)$,
\begin{equation}\label{PI}
    \frac{1}{\mu(B)}\sum_{x\in B} |f(x)-f_B|\mu(x) \le C_p r \left(\frac{1}{\mu(2B)}\sum_{x,y\in 2B}|f(y)-f(x)|^p\mu_{xy}\right)^{\frac{1}{p}}, \tag{$P_p$}
\end{equation}
where $2B=B(x,2r)$ and the average $f_B$ is defined as
\[ f_B = \frac{1}{\mu(B)}\sum_{y \in B} f(y)\mu(y). \]
\end{definition}

As a consequence of \hyperref[PI]{$(P_p)$}, for any function $u$ vanishing outside $B$, we have
\begin{equation*}
    \frac{1}{\mu(B)} \sum_{x \in B} |u(x)| \mu(x) \le C_p' \left( \frac{r^p}{\mu(2B)} \sum_{x,y \in 2B} |\nabla_{xy} u|^p \mu_{xy} \right)^{1/p}.
\end{equation*}

\begin{theorem}[Wiener criterion at infinity]\label{theo_weiner}
Let \((V,\mu)\) be an infinite, connected, and locally finite graph. Assume that $(V,\mu)$ satisfies the volume doubling condition \hyperref[VD]{$\text{(VD)}$} and the Poincar\'e inequality \hyperref[PI]{$(P_p)$}. Let $\Omega \subset V$ be an infinite connected set. If $\Omega$ is $p$-massive, then there exists $x_0\in \Omega$ such that
\begin{align}\label{intr_weiner}
    \sum_{n=1}^{\infty}\left(\frac{cap_{p}(A_{n},B_{n+1})}{cap_{p}(B_{n},B_{n+1})}\right)^{\frac{1}{p-1}}<\infty,
\end{align}
where $B_n = B(x_0, 2^n)$ and $A_n =\Omega^c \cap B_n$. If $(V,\mu)$ additionally satisfies the \hyperref[p0]{$(p_{0})$} condition, then the converse also holds.
\end{theorem}

\begin{figure}[htbp]
    \centering
    \begin{tikzpicture}[scale=0.6]
        
        % \Omega^c path
        \def\omegacpath{ 
            (1.5, -4.8) to[out=110, in=-70] (0.5, -1.8) 
            to[out=110, in=180] (2.8, 1.8) 
            to[out=0, in=200] (4.8, 2.8) -- (4.8, -4.8) -- cycle 
        }

        % \Omega^c  field
        \fill[gray!15] \omegacpath;
        \draw[black, thin] \omegacpath; 

        % B_n path
        \draw[dashed, black] (0,0) circle (1.5);
        \draw[dashed, black] (0,0) circle (3.0);
        \draw[dashed, black, thick] (0,0) circle (4.5);

        % A_n field
        \begin{scope}
            \clip (0,0) circle (3.0);
            \fill[gray!50] \omegacpath;
            \draw[black, thick] \omegacpath;
        \end{scope}

        % x_0
        \filldraw[black] (0,0) circle (2pt);
        \node[below left, font=\normalsize] at (0,0) {$x_0$};

        % B_n 
        \node[black, font=\small] at (135:1.0) {$B_{n-1}$};
        \node[black, font=\small] at (135:2.3) {$B_n$};
        \node[black, font=\small] at (135:3.8) {$B_{n+1}$};

        % \Omega
        \node[font=\Large] at (-2.5, -2.5) {$\Omega$};
        
        % \Omega^c 
        \node[font=\Large, black] at (3.8, -3.2) {$\Omega^c$};

        % A_n 
        \node[font=\Large, black] at (1.8, -0.5) {$A_n$};

    \end{tikzpicture}
    \caption{Region $A_n$ and $B_n$.}
    \label{fig_weiner}
\end{figure}
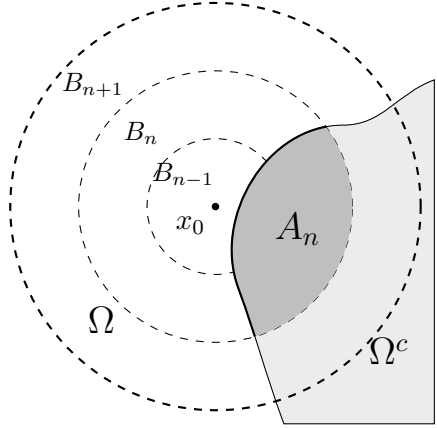

\begin{remark}\rm
Provided that the graph $(V,\mu)$ is not $p$-parabolic, the convergence condition \eqref{intr_weiner} in Theorem \ref{theo_weiner} can be equivalently replaced by
$$
    \sum_{n=1}^{\infty} \left( \frac{r_n^p \operatorname{cap}_{p}(A_n)}{\mu(B_{n})} \right)^{\frac{1}{p-1}} < \infty.
$$
This alternative formulation directly corresponds to the Wiener criterion obtained by It\^o and McKean \cite{IM60} on $\mathbb{Z}^d$ for $d\ge 3$ and $p=2$. Furthermore, when $\partial\Omega$ is finite, $\Omega$ is $p$-massive if and only if it is not $p$-parabolic; see Section \ref{pre}. The $p$-parabolic problem on graphs has been widely investigated; see \cite{HS97phar, HK01, AFS25}.
\end{remark}
The Wiener criterion is accompanied by two structural results that do not require the geometric assumptions \hyperref[VD]{$\text{(VD)}$} and \hyperref[PI]{$(P_p)$}. The first identifies $p$-massiveness with a strengthened form of nonuniqueness for the exterior Dirichlet problem \eqref{dp}.

\begin{theorem}\label{theo_uniq}
Let $(V,\mu)$ be an infinite, connected, and locally finite graph. Then the following statements are equivalent:
\begin{enumerate}
    \item $\Omega \subset V$ is a $p$-massive set.
    \item For some (equivalently, every) $f \in \mathcal{B}(\partial \Omega)$, the Dirichlet problem \eqref{dp} admits two bounded solutions $u$ and $v$ such that $\sup_{\Omega} u \neq \sup_{\Omega} v$.
\end{enumerate}
\end{theorem}

The next theorem characterizes the failure of the $p$-Liouville property through disjoint massive sets. For manifolds, this equivalence goes back to Grigor'yan \cite{gri88} when $p=2$ and to Holopainen \cite{hol94} for $p>1$. In the discrete setting, Holopainen and Soardi \cite{HS97phar} treated the $D_p$-Liouville case for unweighted graphs with bounded degree. Our result extends this correspondence to general weighted graphs.

\begin{theorem}\label{theo_liou}
Let $G = (V, E)$ be an infinite, connected, and locally finite graph. Then $G$ admits a nonconstant bounded $p$-harmonic function (resp., a nonconstant bounded $p$-harmonic function with finite $p$-energy) if and only if there exist two disjoint $p$-massive sets (resp., two disjoint $D_p$-massive sets) in $V$.
\end{theorem}

Finally, we give a discrete analogue of the manifold criterion of Grigor'yan \cite{gri88} and Holopainen \cite{hol94} for $D_p$-massive sets, using tools from nonlinear potential theory on networks.

\begin{theorem}\label{theo_dp}
Let \((V,\mu)\) be an infinite, connected, and locally finite graph. A subset \( \Omega \subset V \) is \( D_p \)-massive if and only if there exists a subset \( \Omega_1 \subset \Omega \) that is not $p$-parabolic and satisfies
\[
\operatorname{cap}_p(\Omega_1, \Omega) < \infty.
\]
\end{theorem}

For linear random walks on graphs, boundary value problems at infinity were studied by Kaimanovich and Woess \cite{KW92}. For related boundary-theoretic perspectives on $p$-harmonic Dirichlet problems and $D_p$-massive sets on graphs, see Kurata \cite{Kur13} and Puls \cite{Puls14}.

\subsection*{Organization of the Paper}
The remainder of the paper is organized as follows. Section \ref{pre} collects preliminaries on $p$-Dirichlet energy and $p$-capacity. Section \ref{lems} proves the technical lemmas needed later. Section \ref{theos} establishes the uniqueness and Liouville results corresponding to Theorem \ref{theo_uniq} and Theorem \ref{theo_liou}. Section \ref{cr_pm} proves the main Wiener criterion at infinity, that is, Theorem \ref{wcm}. Section \ref{cr_dpm} proves the criterion for $D_p$-massive sets, namely Theorem \ref{dpmcr}. Section \ref{example} presents examples illustrating both criteria.

\section{Preliminaries}\label{pre}

In this section we collect the definitions and basic facts on $p$-Dirichlet
energy, $p$-capacity, and $p$-parabolicity that will be used throughout the
paper.

For $p \in (1, \infty)$ and $\Omega \subseteq V$, the $p$-Dirichlet energy of a function $f \in \ell(\Omega)$ is defined by
\begin{equation}
    D_p(f; \Omega) = \frac{1}{2} \sum_{x, y \in \Omega} \mu_{xy} |f(x) - f(y)|^p.
\end{equation}

Then, we introduce the set of admissible functions of \((B_1, B_2; \Omega)\), denoted by $\mathcal{A}(B_1, B_2; \Omega)$:
\begin{equation}
    \mathcal{A}(B_1, B_2; \Omega) := \{ f: \Omega \to \mathbb{R} \mid f|_{B_1} = 1, f|_{B_2} = 0, \text{ and } D_p(f; \Omega) < \infty \}.
\end{equation}

The $p$-effective conductance between $B_1$ and $B_2$ in $\Omega$ is then given by
\begin{equation}
    C_p(B_1, B_2; \Omega) = \inf_{f \in \mathcal{A}(B_1, B_2; \Omega)} D_p(f; \Omega).
\end{equation}
If $\mathcal{A}(B_1, B_2; \Omega) = \emptyset$, we set $C_p(B_1, B_2; \Omega) = \infty$. Clearly, \(C_p(B_1, B_2; \Omega)=C_p(B_2, B_1; \Omega)\). Furthermore, when $\Omega = V$, we simplify the notation to $C_p(B_1, B_2)$.

We further introduce a related quantity. For subsets $K \subset U$, the $p$-capacity of a finite set $K$ relative to $U$ in $\Omega$ is defined as
\begin{equation*}
    \operatorname{cap}_p(K, U;\Omega) = \inf\{D_p(f; \Omega) \mid f \in \ell_0(U), f = 1 \text{ on } K\}.
\end{equation*}

When $U$ is finite, it naturally follows that
\begin{equation}
    \operatorname{cap}_p(K, U;\Omega) = C_p(K, U^c;\Omega).
\end{equation}

If $U$ is infinite and $U \neq \Omega$, let $\{B_n\}_{n=1}^\infty$ be an exhaustion of $\Omega$. The following equalities hold:
\begin{equation*}
    \operatorname{cap}_p(K, U;\Omega) = \lim_{n \to \infty} C_p(K, (B_n \cap U)^c;\Omega) = C_p(K, U^c;\Omega).
\end{equation*}

Finally, we define
\begin{equation*}
    C_p(K, \infty;\Omega) := \lim_{n \to \infty} C_p(K, B_n^c;\Omega),
\end{equation*}
from which it follows that $C_p(K, \infty;\Omega) = \operatorname{cap}_p(K, \Omega;\Omega)$. For convenience, when $\Omega = V$, we simplify the notation to $\operatorname{cap}_p(K)$.

We say that a subset $\Omega \subset V$ is $p$-parabolic if $C_p(K, \infty;\Omega) = 0$ for some (and consequently, any) finite subset $K \subset \Omega$. 
\begin{proposition}\label{prop:parabolic_equiv}
 Let $(V,\mu)$ be an infinite, connected, and locally finite weighted graph. The following statements are equivalent:
\begin{enumerate}
    \item $(V,\mu)$ is $p$-parabolic.
    \item For any finite non-empty subset $A \subset V$, $\operatorname{cap}_p(A) = 0$.
    \item Every positive $p$-superharmonic function on $V$ is constant.
\end{enumerate}
\end{proposition}
\begin{proof}
The equivalence of these statements is a well-known result in nonlinear potential theory on graphs; see \cite{Sal97}, or \cite{AFS25} for a more general setting.
\end{proof}
\begin{remark}
    A subset $\Omega$ being $p$-parabolic is equivalent to the induced subgraph $(\Omega, \mu|_{\Omega \times \Omega})$ being a $p$-parabolic graph.
\end{remark}

While the $p$-effective conductance $C_p(K, U;\Omega)$ is more general than $\operatorname{cap}_p(K, U^c;\Omega)$, as it does not require $K$ to be finite, we will sometimes use these two notations interchangeably in the subsequent text, even when $K$ is infinite.

An immediate consequence of the definition is:
\begin{proposition}
Let $B_1, B_2$ be disjoint non-empty subsets of $\Omega$. The $p$-effective conductance satisfies the following properties:
\begin{enumerate}
    \item[(i)] If $B_1 \subseteq B_1'$, $B_2 \subseteq B_2'$, and $\Omega \subseteq \Omega'$, then $C_p(B_1, B_2; \Omega) \le C_p(B_1', B_2'; \Omega')$.
    \item[(ii)] $C_p(B_1, B_2; \Omega) = C_p(\partial B^o_1, \partial B^o_2; \Omega)$, where $U^o$ denotes the interior of a set $U$.
\end{enumerate}
\end{proposition}

Next, we introduce some well known properties:
\begin{proposition}
Suppose that $C_p(B_1, B_2; \Omega) < \infty$. Then:
\begin{enumerate}
    \item[(i)] There exists a unique function $u \in \mathcal{A}(B_1, B_2; \Omega)$ such that $D_p(u; \Omega) = C_p(B_1, B_2; \Omega)$. We call this unique minimizer the $p$-potential of $(B_1, B_2; \Omega)$.
    \item[(ii)] The function $u$ is a solution to the Dirichlet problem:
    \begin{equation}
        \begin{cases}
            u(x) = 1, & \text{for } x \in B_1, \\
            u(x) = 0, & \text{for } x \in B_2, \\
            \Delta_{\Omega,p} u(x) = 0, & \text{for } x \in \Omega \setminus (B_1 \cup B_2).
        \end{cases}
    \end{equation}
\end{enumerate}
\end{proposition}

\begin{proof}
See \cite[Theorem 2.17]{Bar17} for the case $p=2$. The proof extends directly to general $p$ by observing that while the space of real-valued functions on $\Omega$ equipped with the norm $\|f\|_{o,p} := (D_p(f; \Omega) + |f(o)|^p)^{1/p}$ is no longer a Hilbert space, it remains uniformly convex and therefore reflexive. For analogous arguments, we refer the reader to \cite{yam77, HS97liou, AFS25}.
\end{proof}
\begin{lemma}[Green's formula]\label{gf}
    Let $(V, \mu)$ be a locally finite weighted graph, and let $\Omega \subset V$ be a non-empty finite subset, Then, for any two functions $f, g$ on $V$, the following identity holds:
\begin{equation}
\begin{aligned}
\sum_{x \in \Omega} \Delta_p f(x) g(x) \mu(x) &= -\frac{1}{2} \sum_{x,y \in \Omega} |\nabla_{xy} f|^{p-2}(\nabla_{xy} f) (\nabla_{xy} g) \mu_{xy} \\
&\quad + \sum_{x \in \Omega} \sum_{y \in \partial \Omega} |\nabla_{xy} f|^{p-2}(\nabla_{xy} f) g(x) \mu_{xy}.
\end{aligned}
\end{equation}
\end{lemma}
\begin{proof}
    The proof follows from a direct computation.
\end{proof}

\begin{lemma}[Uniqueness of the Dirichlet problem]\label{duni}
    Let $(V, \mu)$ be a connected, locally finite weighted graph. Let $\Omega$ be a finite subset of $V$ such that $\Omega^c \neq \emptyset$. Then, for any $f \in \ell(\partial\Omega)$, the following Dirichlet problem admits a unique solution:
\begin{equation}
\begin{cases} 
\Delta_p u(x) = 0 & \text{for all } x \in \Omega, \\
u(x) = f(x) & \text{for all } x \in \partial \Omega.
\end{cases}
\end{equation}
\end{lemma}

\begin{lemma}[Comparison principle]\label{cp}
    Let $\Omega \subset V$ be a finite set. If $u$ is $p$-superharmonic and $v$ is $p$-subharmonic in $\Omega$ such that $u \geq v$ on $\partial \Omega$, then $u \geq v$ in $\Omega$.
\end{lemma}

\begin{proof}
      Lemmas \ref{duni} and \ref{cp} correspond to Theorem 3.11 and Theorem 3.14 in \cite{HS97phar}. While the original results were stated for unweighted graphs, the proofs extend directly to the weighted setting.
\end{proof}

\section{Lemmas}\label{lems}

This section gathers the auxiliary estimates needed in the later proofs of the
main theorems. We begin with upper and lower bounds for the $p$-capacity of
balls.

\begin{lemma}\label{cap_ballupb}
  For any $0 < r < R$, define $B_1 = B(x_0, r)$ and $B_2 = B(x_0, R)$, then,
    \begin{equation}
         \operatorname{cap}_{p}(B_1, B_2) \le \frac{\mu(B_2)}{(R-r)^p}.
    \end{equation}
\end{lemma}
\begin{proof}
     Define a cut-off function $\eta$ such that $0 \le \eta \le 1$ as follows:
    \[
    \eta(x) := \begin{cases} 
        1, & \text{if } d(x, x_0) \le r, \\
        0, & \text{if } d(x, x_0) > R, \\
        \frac{R - d(x, x_0)}{R - r}, & \text{otherwise.}
    \end{cases}
    \]
    By construction, $\eta \equiv 1$ on $B_1$, $\operatorname{supp}(\eta) \subset B_2$, and $|\nabla_{xy} \eta| \le 1/(R-r)$. From the definition of $p$-capacity, we have
    \begin{align*}
        \operatorname{cap}_{p}(B_1, B_2) \le \sum_{x,y \in V} |\nabla_{xy} \eta|^p \mu_{xy} \le \frac{\mu(B_2)}{(R-r)^p}.
    \end{align*}
\end{proof}

\begin{lemma}\label{cap_balllob}
    Assume condition \hyperref[VD]{$\text{(VD)}$} and \hyperref[PI]{$(P_p)$} is satisfied on $(V, \mu)$. For any $0 < r < R<2r $,  let $B_1 = B(x_0, r)$ and $B_2 = B(x_0, R)$, then there exists a constant $C_b > 0$ such that 
    \begin{equation}
        C_b \frac{\mu(B_2)}{R^p} \le \operatorname{cap}_{p}(B_1, B_2).
    \end{equation}
\end{lemma}

\begin{proof}
    Let $u$ be the $p$-potential for the capacitor $(B_1, B_2)$. Since $u \equiv 1$ on $B_1$ and $\operatorname{supp}(u) \subset B_2$, we apply \hyperref[VD]{$\text{(VD)}$} and \hyperref[PI]{$(P_p)$}, then
    \begin{align*}
        1 \le \frac{C_D}{\mu(B_2)} \sum_{x \in B_2} u(x) \mu(x) \le C_D C_p' \left( \frac{R^p}{\mu(2B_2)} \sum_{x,y \in 2B_2} |\nabla_{xy} u|^p \mu_{xy} \right)^{1/p}.
    \end{align*}
    Raising both sides to the power $p$, we obtain 
    \begin{align*}
        C_b \frac{\mu(B_2)}{R^p} \le \sum_{x,y \in V} |\nabla_{xy} u|^p \mu_{xy} = \operatorname{cap}_{p}(B_1, B_2),
    \end{align*} 
    where the constant $C_b = (C_D C_p' )^{-p}$.
\end{proof}

\begin{lemma}\label{pot_app}
Let \(K\) and \(\Omega\) be disjoint non-empty subsets of \(V\) with \(C_p(K,\Omega)<\infty\). Let \(\{B_n\}_{n=1}^\infty\) be an exhaustion of \(V\), and for each \(n\), let \(u_n\) be the \(p\)-potential associated with \(C_p(K,\Omega;B_n)\). Then \(u_n\) converges pointwise on \(V\) to the \(p\)-potential \(u\) of \((K,\Omega)\). Moreover,
\[
    \lim_{n\to \infty} C_p(K,\Omega;B_n)= C_p(K,\Omega).
\]
\end{lemma}

\begin{proof}
Since \(0\le u_n\le 1\) on \(V\) for every \(n\), a diagonal argument yields a subsequence \(\{u_{n_j}\}\) and a function \(v\) such that \(u_{n_j}(x)\to v(x)\) for every \(x\in V\).

Clearly, \(v=1\) on \(K\) and \(v=0\) on \(\Omega\). By Fatou's lemma,
\[
    D_p(v)\le \liminf_{j\to\infty} D_p(u_{n_j})
    = \liminf_{j\to\infty} C_p(K,\Omega;B_{n_j})
    \le C_p(K,\Omega),
\]
where we used the fact \(D_p(f;B_n)\le(D_p(f) \), \(\mathcal{A}(K, \Omega)\subset \mathcal{A}(K, \Omega; B_n) \) and
\[
   C_p(K, \Omega;B_n ) = \inf_{f \in \mathcal{A}(K, \Omega; B_n)} D_p(f; \Omega).
\]
Thus \(v\) is admissible for \(C_p(K,\Omega)\), the reverse inequality \(C_p(K,\Omega)\le D_p(v)\) also holds. Therefore \(D_p(v)=C_p(K,\Omega)\), and by uniqueness of the \(p\)-potential, \(v=u\).

Since every pointwise convergent subsequence of \(\{u_n\}\) has the same limit \(u\), the whole sequence \(\{u_n\}\) converges pointwise to \(u\).

\end{proof}

\begin{lemma}\label{lem_eng}
For two disjoint non-empty subsets \(K \) and \( \Omega\),  assume that \( C_p(K, \Omega)<\infty\) and \(u\) is the $p$-potential for the capacitor \((K, \Omega)\). Then, for almost every \(t\in(0,1)\), and also for \(t=0\) and \(t=1\), we have  
\begin{align}\label{cap_eng}
    C_p(K, \Omega)= \sum_{xy \in \partial_e\Gamma_t} |\nabla_{xy} u|^{p-1} \mu_{xy},
\end{align}
where \(\Gamma_t=\{x\in V: u(x)> t\}\) for any \(0 \le t < 1\), and \(\Gamma_1=K\).
\end{lemma}

Lemma \ref{lem_eng} admits a direct interpretation in the context of non-linear electrical networks. The $p$-capacity $C_p(K, \Omega)$ represents the $p$-effective conductance between $K$ and $\Omega$. Accordingly, the term
$$
    \sum_{xy \in \partial_e\Gamma_t} |\nabla_{xy} u|^{p-1} \mu_{xy}
$$
corresponds to the total effective $p$-current flowing across the edge cut $\partial_e\Gamma_t$. The equality \eqref{cap_eng} reflects the conservation of current: due to the $p$-harmonicity of the potential $u$, the total current is conserved across almost every intermediate level set, and in particular across the endpoint cuts corresponding to \(t=0\) and \(t=1\).

\begin{proof}
If \(K=\Omega^c\), then \eqref{cap_eng} is immediate from the definition of \(C_p(K,\Omega)\). For \(0\le t\le1\), define
\[
    h_{p,t}(u):=\sum_{xy\in\partial_e\Gamma_t} |\nabla_{xy}u|^{p-1}\mu_{xy}.
\]

\textbf{Claim 1.}
\begin{equation}\label{cafm_inf}
	D_p(u)=\int_0^1 h_{p,t}(u)\,dt.
\end{equation}

Indeed, for each edge \(xy\in E\), let
\[
I_{xy}:=[\min\{u(x),u(y)\},\max\{u(x),u(y)\}).
\]
Since \(xy\in \partial_e\Gamma_t\) if and only if \(t\in I_{xy}\), Fubini's theorem gives
\begin{align}\label{cafm}
	\int_0^1 h_{p,t}(u)\,dt
	&= \int_0^1 \sum_{xy\in E} |\nabla_{xy}u|^{p-1}\mathbf{1}_{I_{xy}}(t)\mu_{xy}\,dt \\
	&= \sum_{xy\in E} |\nabla_{xy}u|^{p-1}\mu_{xy}\,|I_{xy}| \\
	&= \sum_{xy\in E} |\nabla_{xy}u|^p\mu_{xy}
	= D_p(u).
\end{align}

\smallskip
\textbf{Claim 2.} For every \(t\in(0,1)\), we have \(h_{p,t}(u)\le D_p(u)\).

Assume first that \(\Omega^c\) is finite. Fix \(0<s<t<1\) and set
\[
    L^-:=\{x\in V:\ u(x)\le s\},\qquad
    L:=\{x\in V:\ s<u(x)\le t\},\qquad
    L^+:=\{x\in V:\ u(x)>t\}.
\]
Since \(L\subset \Omega^c\) and \(\Omega^c\) is finite, the set \(L\) is finite. Define the oriented \(p\)-current by
\[
    J(x,y):=-\mu_{xy}|\nabla_{xy}u|^{p-2}\nabla_{xy}u
    =\mu_{xy}|u(x)-u(y)|^{p-2}(u(x)-u(y)).
\]
Then \(J(y,x)=-J(x,y)\), and for every \(x\in \Omega^c\setminus K\),
\[
    \sum_{y\sim x}J(x,y)=-\mu(x)\Delta_pu(x)=0.
\]
Summing over \(x\in L\) and canceling the contributions of edges with both endpoints in \(L\), we obtain
\[
    0=\sum_{\substack{x\in L\\ y\in L^-}}J(x,y)
      +\sum_{\substack{x\in L\\ y\in L^+}}J(x,y).
\]
Moreover, if \(x\in L\) and \(y\in L^-\), then \(u(x)>u(y)\), so \(J(x,y)=|J(x,y)|\); if \(x\in L\) and \(y\in L^+\), then \(u(x)<u(y)\), so \(J(x,y)=-|J(x,y)|\). Hence
\begin{equation}\label{eq:strip-balance}
    \sum_{\substack{x\in L\\ y\in L^-}}|J(x,y)|
    =
    \sum_{\substack{x\in L\\ y\in L^+}}|J(x,y)|.
\end{equation}

Set
\[
    C_{s,t}:=\sum_{\substack{x\in L^+\\ y\in L^-}} |J(x,y)|.
\]
Since \(\Gamma_s=L\cup L^+\) and \(\Gamma_t=L^+\), we have
\[
    h_{p,s}(u)=C_{s,t}+\sum_{\substack{x\in L\\ y\in L^-}} |J(x,y)|
\]
and
\[
    h_{p,t}(u)=C_{s,t}+\sum_{\substack{x\in L\\ y\in L^+}} |J(x,y)|.
\]
Combining these identities with \eqref{eq:strip-balance}, we obtain
\[
    h_{p,s}(u)=h_{p,t}(u).
\]
Thus \(h_{p,t}(u)\) is constant on \(t\in (0,1)\), and \eqref{cafm} shows that this constant equals \(D_p(u)\). Since only finitely many edges are involved, letting \(s\downarrow0\) and \(t\uparrow1\) also gives
\[
    h_{p,0}(u)=h_{p,1}(u)=D_p(u).
\]
Hence \eqref{cap_eng} holds for every \(t\in[0,1]\) whenever \(\Omega^c\) is finite.

Now, assume \(\Omega^c\) is infinite.

Let \(\{B_n\}\) be an exhaustion of \(V\), and let \(u_n\) be the \(p\)-potential associated with \(C_p(K,\Omega;B_n)\). By Lemma \ref{pot_app},
\[
    u_n(x)\to u(x)\quad \text{for every }x\in V,
    \qquad
    D_p(u_n)\to D_p(u)=C_p(K,\Omega).
\]
For each \(n\), let
\[
    \Gamma_t^n:=\{x\in V: u_n(x)>t\}.
\]
By the finite case, for every \(n\) and every \(t\in[0,1]\),
\[
    D_p(u_n)=\sum_{xy\in \partial_e\Gamma_t^n} |\nabla_{xy}u_n|^{p-1}\mu_{xy}.
\]

Fix \(t\in(0,1)\), and let \(F\) be any finite subset of \(\partial_e\Gamma_t\). For each \(xy\in F\), orient the edge so that \(u(x)>t\ge u(y)\). Since \(F\) is finite, there exists \(t'>t\) such that
\[
    u(x)>t' \quad\text{and}\quad u(y)<t'
    \qquad\text{for every }xy\in F.
\]
By the pointwise convergence \(u_n\to u\), there exists \(n_0\) such that for every \(n\ge n_0\),
\[
    u_n(x)>t' \quad\text{and}\quad u_n(y)<t'
    \qquad\text{for every }xy\in F.
\]
Hence \(F\subset \partial_e\Gamma_{t'}^n\) for all \(n\ge n_0\). Since \(F\) is finite,
\[
    \sum_{xy\in F} |\nabla_{xy}u_n|^{p-1}\mu_{xy}\longrightarrow
    \sum_{xy\in F} |\nabla_{xy}u|^{p-1}\mu_{xy}.
\]
Therefore,
\begin{align*}
    \sum_{xy\in F} |\nabla_{xy}u|^{p-1}\mu_{xy}
    &= \lim_{n\to\infty}\sum_{xy\in F} |\nabla_{xy}u_n|^{p-1}\mu_{xy} \\
    &\le \lim_{n\to\infty}\sum_{xy\in \partial_e\Gamma_{t'}^n} |\nabla_{xy}u_n|^{p-1}\mu_{xy} \\
    &= \lim_{n\to\infty} D_p(u_n)
     = D_p(u).
\end{align*}
Letting \(F\uparrow \partial_e\Gamma_t\), we obtain \(h_{p,t}(u)\le D_p(u)\), which proves Claim 2.

Combining \eqref{cafm_inf} with Claim 2, we conclude that
\[
    h_{p,t}(u)=D_p(u)\qquad\text{for almost every }t\in(0,1).
\]

To treat the endpoint \(t=1\), note that for each \(n\),
\[
    D_p(u_n)=\sum_{xy\in\partial_e K} |\nabla_{xy}u_n|^{p-1}\mu_{xy}.
\]
Passing to the limit and using Lemma \ref{pot_app}, we obtain
\[
    h_{p,1}(u)=D_p(u).
\]
Finally, the function \(1-u\) is the \(p\)-potential of \((\Omega,K)\), and the same argument applied to \(1-u\) yields
\[
    h_{p,0}(u)=D_p(u).
\]
This completes the proof.
\end{proof}

\begin{lemma}\label{cap_pt}
    For \(K \subset U\), assume that \(u\) is the $p$-potential function of \((K,U)\),  let \(\sigma=-\Delta_p u\), then
    \begin{align} \label{cro1}
        \operatorname{cap}_{p}(K, U)= \sigma(K).
    \end{align}
\end{lemma}
\begin{proof} 
Since \(u\equiv 1\) on \(K\), we have \(\nabla_{xy}u=0\) whenever \(x,y\in K\). Hence
\begin{align*}
    \sigma(K)
    &= \sum_{x\in K}\sigma(x) m(x)
     = -\sum_{x\in K}\Delta_p u(x) m(x) \\
    &= -\sum_{x\in K}\sum_{y\sim x}\mu_{xy} |\nabla_{xy}u|^{p-2}\nabla_{xy}u \\
    &= -\sum_{x\in K}\sum_{y\in \partial K}\mu_{xy} |\nabla_{xy}u|^{p-2}\nabla_{xy}u.
\end{align*}
For \(x\in K\) and \(y\in \partial K\), we have \(u(x)=1\ge u(y)\), so
\[
    -|\nabla_{xy}u|^{p-2}\nabla_{xy}u=|\nabla_{xy}u|^{p-1}.
\]
Therefore
\[
    \sigma(K)=\sum_{x\in K}\sum_{y\in \partial K} |\nabla_{xy}u|^{p-1}\mu_{xy}=h_{p,1}(u).
\]
Applying Lemma \ref{lem_eng} with \(t=1\), we obtain
\[
    \sigma(K)=\operatorname{cap}_p(K,U).
\]
This proves \eqref{cro1}.
\end{proof}

\begin{lemma}\label{cap_lvs}
    Let $K \subset U$ and let $u$ be the $p$-potential function of the condenser $(K, U)$. For any $0 < \lambda < 1$ such that the level set $\Gamma_\lambda := \{x \in V : u(x) > \lambda\}$ satisfies $\overline{\Gamma_\lambda} \subset U$, the following holds, 
  \begin{align}
       \lambda^{p-1}\operatorname{cap}_{p}(\Gamma_\lambda, U)\le \operatorname{cap}_{p}(K, U) \le  \lambda^{p-1} \operatorname{cap}_{p}(\overline{\Gamma_\lambda}, U).
  \end{align}
\end{lemma}

\begin{proof}
Let \(v\) be the \(p\)-potential function of \((\Gamma_\lambda, U)\). Then \(v=1\) on \(\Gamma_\lambda\), \(v=0\) on \(U^c\), and \(v\) is \(p\)-harmonic on \(U\setminus \Gamma_\lambda\). Since \(u/\lambda\ge 1\) on \(\Gamma_\lambda\), \(u/\lambda=0\) on \(U^c\), and \(u/\lambda\) is also \(p\)-harmonic on \(U\setminus \Gamma_\lambda\), the comparison principle applied on an exhaustion of \(U\setminus \Gamma_\lambda\) yields
\[
    u\ge \lambda v \qquad \text{on } U\setminus \Gamma_\lambda.
\]
Because \(\overline{\Gamma_\lambda}\subset U\), every vertex \(x\in U\) adjacent to \(\partial U\) lies in \(U\setminus \Gamma_\lambda\). Applying Lemma \ref{lem_eng} with \(t=0\) to \(u\) and \(v\), we obtain
\begin{align*}
    \lambda^{p-1}\operatorname{cap}_{p}(\Gamma_\lambda, U)
    &= \sum_{x \in U} \sum_{y \in \partial U} (\lambda v(x))^{p-1} \mu_{xy} \\
    &\le \sum_{x \in U} \sum_{y \in \partial U} u(x)^{p-1} \mu_{xy}
     = \operatorname{cap}_{p}(K, U).
\end{align*}

For the reverse inequality, let \(v'\) be the \(p\)-potential function of \((\overline{\Gamma_\lambda}, U)\). Then \(v'=1\) on \(\overline{\Gamma_\lambda}\), \(v'=0\) on \(U^c\), and \(v'\) is \(p\)-harmonic on \(U\setminus \overline{\Gamma_\lambda}\). On \(\partial U\), both \(u\) and \(\lambda v'\) vanish. On the inner boundary \(\partial\Gamma_\lambda\), we have \(u\le \lambda=\lambda v'\). Therefore, applying the comparison principle on an exhaustion of \(U\setminus \overline{\Gamma_\lambda}\), we obtain
\[
    u\le \lambda v' \qquad \text{on } U\setminus \overline{\Gamma_\lambda}.
\]
Again using Lemma \ref{lem_eng} with \(t=0\), we deduce
\begin{align*}
    \operatorname{cap}_{p}(K, U)
    &= \sum_{x \in U} \sum_{y \in \partial U} u(x)^{p-1} \mu_{xy} \\
    &\le \sum_{x \in U} \sum_{y \in \partial U} (\lambda v'(x))^{p-1} \mu_{xy}
     = \lambda^{p-1} \operatorname{cap}_{p}(\overline{\Gamma_\lambda}, U).
\end{align*}
This completes the proof.
\end{proof}

\section{Theorems on Uniqueness and Liouville Property}\label{theos}

We now prove the structural results announced in the introduction: the
equivalence between $p$-massiveness and strengthened nonuniqueness for the
exterior Dirichlet problem, and the characterization of bounded nonconstant
$p$-harmonic functions by disjoint massive sets.

In this section, for convenience, we call a function $u \in \ell(\overline{\Omega})$ admissible on $\Omega$ if $u=0$ on $\partial \Omega$ and $0 < u < 1$ on $\Omega$. 

\begin{theorem}\label{theo_uniqp}
Let $G = (V, E)$ be an infinite, connected, and locally finite graph. Then the following statements are equivalent:
\begin{enumerate}
    \item $\Omega \subset V$ is a $p$-massive set.
    \item For some (equivalently, every) \(f \in \mathcal{B}(\partial \Omega)\), the Dirichlet problem~\eqref{dp} admits two bounded solutions \(u\) and \(v\) such that \(\sup_{\Omega} u \neq \sup_{\Omega} v\).
\end{enumerate}
\end{theorem}

\begin{proof}
(1) $\Rightarrow$ (2):  
Suppose \(\Omega\) is \(p\)-massive, and let \(h\) be an admissible \(p\)-harmonic function on \(\Omega\). We show that the conclusion holds for every \(f\in \mathcal{B}(\partial \Omega)\). Let \(v\) be any bounded solution to \eqref{dp} with boundary datum \(f\). Replacing \(v\) and \(f\) by \(v-m\) and \(f-m\), where \(m:=\inf_{\partial\Omega} f\), we may assume that
\[
    f\ge 0 \qquad \text{on } \partial\Omega.
\]
By the comparison principle, \(v\ge 0\) in \(\Omega\). Set
\[
    M:=\sup_{x\in \Omega} v(x).
 \]
Choose \(c> M/\sup_{\Omega} h\). Then there exists \(x_0\in \Omega\) such that
\[
    ch(x_0)>M\ge v(x_0).
\]

For each \(k\ge 1\), let \(\Omega_k:=B_k\cap \Omega\), and define
\[
    F_k(x):=
    \begin{cases}
        f(x), & x\in \partial \Omega\cap \partial \Omega_k,\\
        c, & x\in \partial \Omega_k\setminus \partial \Omega.
    \end{cases}
\]
Let \(v_k\) be the unique solution of
\[
    \begin{cases}
        \Delta_p w = 0, & \text{in } \Omega_k, \\
        w = F_k, & \text{on } \partial \Omega_k.
    \end{cases}
\]
Since \(v\le F_k\) on \(\partial \Omega_k\), the comparison principle gives \(v\le v_k\) in \(\Omega_k\). Moreover, \(ch=0\) on \(\partial \Omega\) and \(0<h<1\) in \(\Omega\), so \(ch\le F_k\) on \(\partial \Omega_k\), and hence
\[
    ch\le v_k\le c \qquad \text{in } \Omega_k.
\]
By a diagonal argument, \(\{v_k\}\) has a pointwise convergent subsequence with limit \(v'\). Then \(v'\) is a bounded solution of \eqref{dp}, and
\[
    v'(x_0)\ge ch(x_0)>M=\sup_{\Omega} v.
\]
Hence
\[
    \sup_{\Omega} v' > \sup_{\Omega} v,
\]
so \(v\) and \(v'\) satisfy (2). This proves the statement for every \(f\in \mathcal{B}(\partial \Omega)\).

\medskip

(2) $\Rightarrow$ (1):  
Assume that, for some \(f\in \mathcal{B}(\partial \Omega)\), the Dirichlet problem \eqref{dp} admits two bounded solutions \(u\) and \(v\) such that
\[
    \sup_{\Omega} u \neq \sup_{\Omega} v.
\]
Replacing \(u\) and \(v\) if necessary, we may assume that
\[
    \sup_{\Omega} u > \sup_{\Omega} v.
\]
Choose \(x_0\in \Omega\) such that
\[
    u(x_0)>\sup_{\Omega} v.
\]
Then \(x_0\in A:=\{x\in \Omega:\ u(x)>v(x)\}\). Let \(A_0\) be the connected component of \(A\) containing \(x_0\). If \(A_0\) were finite, then \(u\le v\) on \(\partial A_0\), and the comparison principle would imply \(u\le v\) in \(A_0\), a contradiction. Thus \(A_0\) is infinite.

Since \(A_0\) is countable, we may choose
\[
    a\in \bigl(\sup_{\Omega} v,u(x_0)\bigr)\setminus u(A_0).
\]
Let \(C\) be the connected component containing \(x_0\) of the superlevel set
\[
    \{x\in A_0:\ u(x)>a\}.
\]
Because \(u\le v\le \sup_{\Omega} v<a\) on \(\partial A_0\), and \(a\notin u(A_0)\), every point of \(\partial C\) satisfies \(u<a\). If \(C\) were finite, then comparing \(u\) with the constant function \(a\) on \(C\) would give \(u\le a\) in \(C\), a contradiction. Thus \(C\) is infinite.

Set
\[
    b:=\sup_C u.
\]
For each integer \(k\ge 1\), let \(C_k:=B_k\cap C\), and define the boundary function \(H_k\colon \partial C_k\to \mathbb{R}\) by
\[
    H_k(x)=
    \begin{cases}
        a,& \text{if } x\in \partial C\cap \partial C_k,\\
        u(x),& \text{if } x\in \partial C_k\setminus \partial C.
    \end{cases}
\]
Consider the Dirichlet problem
\[
    \begin{cases}
        \Delta_p z=0,& \text{in }C_k,\\
        z=H_k,& \text{on }\partial C_k.
    \end{cases}
\]
By Lemma~\ref{duni}, it admits a unique solution, denoted by \(z_k\). Since \(u\le H_k\) on \(\partial C_k\), the comparison principle yields
\[
    u\le z_k \qquad \text{in }C_k.
\]
Moreover, \(a\le H_k\le b\) on \(\partial C_k\), and therefore
\[
    a<u\le z_k\le b \qquad \text{in }C_k.
\]
By a diagonal argument, \(\{z_k\}\) admits a pointwise convergent subsequence \(\{z_{k_i}\}\). Define
\[
    h:=\lim_{i\to\infty}\frac{z_{k_i}-a}{b-a}.
\]
Then \(h\) is \(p\)-harmonic on \(C\), and for every \(x\in C\),
\[
    0<\frac{u(x)-a}{b-a}\le h(x)\le 1.
\]
If \(y\in \partial C\), then for all sufficiently large \(i\) we have \(y\in \partial C\cap \partial C_{k_i}\), so \(z_{k_i}(y)=a\), and hence \(h(y)=0\). Thus \(h=0\) on \(\partial C\). Since \(C\) is connected and \(h\not\equiv 1\) on \(C\), the strong maximum principle implies
\[
    0<h<1 \qquad \text{in }C.
\]
Hence \(h\) is admissible on \(C\), so \(C\) is \(p\)-massive. Thus \(\Omega\) is \(p\)-massive by Remark~\ref{rem_mon_mass}.
\end{proof}

\begin{theorem}\label{theo_lioup}
Let $G = (V, E)$ be an infinite, connected, and locally finite graph. Then $G$ admits a nonconstant bounded $p$-harmonic function (resp., with finite $p$-energy) if and only if there exist two disjoint $p$-massive (resp., $D_p$-massive) subsets of $V$.
\end{theorem}

\begin{proof}
Suppose $u$ is a nonconstant bounded $p$-harmonic function on $V$. Since $V$ is countable, we may choose
\[
    a \in (\inf_{x \in V} u(x), \sup_{x \in V} u(x)) \setminus u(V).
\]
Define the level sets
$$
\Omega_1 := \{x \in V : u(x) > a\}, \qquad
\Omega_2 := \{x \in V : u(x) < a\}.
$$
These sets are clearly disjoint and nonempty due to the choice of $a$ and the non-triviality of $u$.

Applying the construction used in the proof of Theorem~\ref{theo_uniqp} to the bounded \(p\)-harmonic function \(u\) on \(\Omega_1\) with boundary level \(a\), we conclude that \(\Omega_1\) is \(p\)-massive. Likewise, applying the same argument to \(-u\) on \(\Omega_2\), we obtain that \(\Omega_2\) is \(p\)-massive.

We now prove that if $D_p(u) < \infty$, then $\Omega_1$ and $\Omega_2$ are also $D_p$-massive. For this, keep the same notation as above: for each integer $k \geq 1$, let $\Omega_1^k := B_k \cap \Omega_1$, and define the boundary function $F_k \colon \partial \Omega_1^k \to \mathbb{R}$ by
$$
F_k(x) =
\begin{cases}
a, & \text{if } x \in \partial \Omega_1 \cap \partial \Omega_1^k, \\
u(x), & \text{if } x \in \partial \Omega_1^k \setminus \partial \Omega_1.
\end{cases}
$$

Consider the Dirichlet problem
$$
\begin{cases}
\Delta_p w = 0, & \text{in } \Omega_1^k, \\
w = F_k, & \text{on } \partial \Omega_1^k.
\end{cases}
$$
By Lemma \ref{duni}, this problem admits a unique solution, denoted by $v_k$. Applying the comparison principle, we have
$$
a \leq u(x) \leq v_k(x) \leq b := \sup_{y \in \Omega_1} u(y), \quad \text{for all } x \in \Omega_1^k.
$$

For convenience, define the sets
$\Gamma_1^k :=  \partial \Omega_1^k \cap \partial \Omega_1$,
$\Gamma_2^k := \partial \Omega_1^k \setminus \partial \Omega_1$,
and
$$E_k = \{xy \in E \mid x \in \Omega_1^k, y \in \overline{\Omega_1^k}\}.$$
Applying the discrete Green formula in Lemma \ref{gf} to $v_k$, we compute the
$p$-Dirichlet energy:
\begin{align*}
 D_k &:= \sum_{xy \in E_k} |\nabla_{xy} v_k|^p \mu_{xy} \\
 &= \sum_{x \in \Omega_1^k} \sum_{y \in \partial\Omega_1^k} |\nabla_{xy} v_k|^{p-2} (\nabla_{xy} v_k)^2 \mu_{xy} + \sum_{x \in \Omega_1^k} \sum_{y \in \Omega_1^k} |\nabla_{xy} v_k|^{p} \mu_{xy} \\
 &= \sum_{x \in \Omega_1^k} \sum_{y \in \partial\Omega_1^k} |\nabla_{xy} v_k|^{p-2} (\nabla_{xy} v_k) v_k(y) \mu_{xy} - \sum_{x \in \Omega_1^k} v_k(x) \Delta_p v_k(x) \mu(x) \\
 &= \sum_{x \in \Omega_1^k} \sum_{y \in \partial\Omega_1^k} |\nabla_{xy} v_k|^{p-2} (\nabla_{xy} v_k) v_k(y) \mu_{xy}.
\end{align*}

Using the definitions of $\Gamma_1^k$ and $\Gamma_2^k$, we split the sum:
\begin{align}\label{EE1}
D_k
&= a \sum_{x \in \Omega_1^k} \sum_{y \in \Gamma_1^k} |\nabla_{xy} v_k|^{p-2} (\nabla_{xy} v_k) \mu_{xy} 
+ \sum_{x \in \Omega_1^k} \sum_{y \in \Gamma_2^k} |\nabla_{xy} v_k|^{p-2} (\nabla_{xy} v_k) u(y) \mu_{xy}.
\end{align}

On the other hand, since $v_k$ is $p$-harmonic on $\Omega_1^k$, we have
$$
\sum_{x \in \Omega_1^k} \Delta_p v_k(x) \mu(x) = 0.
$$
From the definition of $\Delta_p$, 
\begin{align*}
0 &= \sum_{x \in \Omega_1^k} \sum_{y \sim x} \mu_{xy} |\nabla_{xy} v_k|^{p-2} \nabla_{xy} v_k \\
&= \sum_{x,y \in \Omega_1^k} \mu_{xy} |\nabla_{xy} v_k|^{p-2} \nabla_{xy} v_k + \sum_{x \in  \Omega_1^k}\sum_{y \in \partial \Omega_1^k} \mu_{xy} |\nabla_{xy} v_k|^{p-2} \nabla_{xy} v_k.
\end{align*}
Since the first term is antisymmetric in $x$ and $y$, it vanishes, and we obtain
$$
\sum_{x \in \Omega_1^k} \sum_{y \in \Gamma_1^k} \mu_{xy} |\nabla_{xy} v_k|^{p-2} \nabla_{xy} v_k = - \sum_{x \in \Omega_1^k} \sum_{y \in \Gamma_2^k} \mu_{xy} |\nabla_{xy} v_k|^{p-2} \nabla_{xy} v_k.
$$

Substituting this into equation \eqref{EE1}, we get
$$
D_k = \sum_{x \in \Omega_1^k} \sum_{y \in \Gamma_2^k} |\nabla_{xy} v_k|^{p-2} (\nabla_{xy} v_k)(u(y) - a) \mu_{xy}.
$$

Now consider the Dirichlet energy of $u$. Using a similar computation, we obtain
\begin{align*}
D_p(u)
&\geq \sum_{xy \in E_k} |\nabla_{xy} u|^p \mu_{xy} \\
&= \sum_{x \in \Omega_1^k} \sum_{y \in \Gamma_1^k} |\nabla_{xy} u|^{p-2} (\nabla_{xy} u) u(y) \mu_{xy}
+  \sum_{x \in \Omega_1^k} \sum_{y \in \Gamma_2^k} |\nabla_{xy} u|^{p-2} (\nabla_{xy} u) u(y) \mu_{xy} \\
&\geq  \sum_{x \in \Omega_1^k} \sum_{y \in \Gamma_2^k} |\nabla_{xy} u|^{p-2} (\nabla_{xy} u) (u(y) - a) \mu_{xy},
\end{align*}
where we used the fact that $u(y) \leq a$ and $u(x) > a$ for all $x \in \Omega_1$ and $y \in \partial \Omega_1$ such that $\mu_{xy} > 0$.

Since $u \leq v_k$ on $\partial \Omega_1^k \cap \partial \Omega_1$ and $v_k = u$ on $\partial \Omega_1^k \setminus \partial \Omega_1$, we have $v_k \geq u$ on $\Omega_1^k$, which implies that
\[
|\nabla_{xy} v_k|^{p-2} (\nabla_{xy} v_k) \mu_{xy} \leq |\nabla_{xy} u|^{p-2} (\nabla_{xy} u) \mu_{xy}, \quad \text{for all } x \in \Omega_1^k \text{ and } y \in \Gamma_2^k.
\]

Therefore,
\begin{align*}
D_p(u) 
&\geq \sum_{x \in \Omega_1^k} \sum_{y \in \Gamma_2^k} |\nabla_{xy} u|^{p-2} (\nabla_{xy} u) (u(y) - a) \mu_{xy} \\
&\geq \sum_{x \in \Omega_1^k} \sum_{y \in \Gamma_2^k} |\nabla_{xy} v_k|^{p-2} (\nabla_{xy} v_k) (u(y) - a) \mu_{xy} = D_k.
\end{align*}

Finally, by Fatou's lemma, we obtain
$$
D_p(v) \leq \liminf_{i \to \infty} \frac{D_{k_i}}{(b-a)^p} \leq \frac{D_p(u)}{(b-a)^p} < \infty,
$$
which shows that $v$ has finite $p$-energy. Therefore, $\Omega_1$ is $D_p$-massive. The same argument applies to $\Omega_2$.
 
Conversely, assume that $\Omega$ and $W$ are two disjoint $p$-massive subsets of $V$, with corresponding admissible $p$-harmonic functions $u$ and $w$, respectively. Let $\{B_k\}$ be an exhaustion of $V$. Define
\[
\Omega_m := B_m \cup \Omega, \quad \text{and} \quad \Omega_m^k := B_k \cap \Omega_m.
\]

As before, define the boundary function
\[
F(x) :=
\begin{cases}
u(x), & \text{if } x \in \Omega, \\
0, & \text{if } x \in V \setminus \Omega.
\end{cases}
\]
We consider the Dirichlet problem
\[
\begin{cases}
\Delta_p w = 0, & \text{in } \Omega_m^k, \\
w = F, & \text{on } \partial \Omega_m^k,
\end{cases}
\]
and let $v_m^k$ denote its unique solution.

By the comparison principle on the region $\Omega_m^k$, we have $0 \leq v_m^k \leq 1$ in $\Omega_m^k$. Since $v_m^k \geq u$ on $\partial (\Omega_m^k \cap \Omega)$, it follows that $v_m^k \geq u$ on $\Omega_m^k \cap \Omega$. Furthermore, applying the comparison principle again on $\Omega_m^{k-1}$, we obtain the monotonicity $v_m^{k-1} \leq v_m^k$. We then define the limit
\[
v_m := \lim_{k \to \infty} v_m^k.
\]

Clearly, $0 \leq v_m \leq 1$ on $\Omega_m$, and $v_m \geq u$ on $\Omega$. Moreover, $v_m$ satisfies
\[
\begin{cases}
\Delta_p v_m = 0, & \text{in } \Omega_m, \\
v_m = F, & \text{on } \partial \Omega_m.
\end{cases}
\]

Assume $m$ is large enough so that $B_m \cap W \neq \emptyset$. Then, applying the comparison principle to the region $B_m \cap W$, and noting that $F = 0$ on $W$, we obtain $v_m \leq 1 - w$. We can then define the limit
\[
v := \lim_{m \to \infty} v_m.
\]
It follows that $v$ is a bounded $p$-harmonic function on $V$ with $v \geq u$ on $\Omega$ and $v \leq 1 - w$ on $W$, meaning $v$ is not a constant function.

It remains to prove that if $D_p(u) < \infty$ and $D_p(w) < \infty$, then
$D_p(v) < \infty$. Applying the discrete Green formula in Lemma \ref{gf} as above, we obtain
\begin{align*}
D_p(v_m^k)
&= \sum_{x \in \Omega_m^k} \sum_{ y \in \partial \Omega_m^k \cap \Omega} |\nabla_{xy} v_m^k|^{p-2} (\nabla_{xy} v_m^k) u(y) \mu_{xy} \\
&\leq \sum_{x \in \Omega_m^k} \sum_{ y \in \partial \Omega_m^k \cap \Omega} |\nabla_{xy} u|^{p-2} (\nabla_{xy} u) u(y) \mu_{xy} \leq D_p(u) < \infty,
\end{align*}
which implies that $D_p(v) \leq D_p(u) < \infty$ by Fatou's lemma, completing the proof.
\end{proof}

\begin{remark}\label{rem_mon_mass}\rm
Note that the proof of Theorem \ref{theo_lioup} establishes the following property: if $\Omega_1 \subset \Omega_2 \subset V$, then the $p$-massiveness (resp., $D_p$-massiveness) of $\Omega_1$ implies the $p$-massiveness (resp., $D_p$-massiveness) of $\Omega_2$.
\end{remark}
   
\section{Criterion for \texorpdfstring{$p$}{p}-Massive Sets}\label{cr_pm}

This section is devoted to the proof of the Wiener criterion at infinity for
$p$-massive sets. We first reformulate massiveness in terms of $p$-thickness
and then derive the capacitary criterion from this point of view.

\begin{definition}
    We say that a set $A \subset V$ is $p$-thick at a point $x_0 \in V$ if, for any positive $p$-superharmonic function $u$ on $V$, the condition $u \ge 1$ on $A$ implies that $u(x_0) \ge 1$.
\end{definition}

\begin{proposition}
The following statements are equivalent:
\begin{enumerate}
    \item $\Omega$ is a $p$-massive set.
    \item There exists some $x_0 \in \Omega$ such that $\Omega^c$ (equivalently, $\partial \Omega$) is not $p$-thick at $x_0$.
    \item If $\Omega$ is connected, then $\Omega^c$ (equivalently, $\partial \Omega$) is not $p$-thick at any point in $\Omega$.
\end{enumerate}
\end{proposition}

\begin{proof}
    $(1) \Rightarrow (2)$: This implication follows immediately from the definitions.
    
    $(2) \Rightarrow (3)$: Assume there exists a positive $p$-superharmonic function $u$ on $V$ such that $u \ge 1$ on $\Omega^c$ and $u(x_0) < 1$ for some $x_0 \in \Omega$.
    
    Similar to the argument in Section \ref{theos}, for each integer $k \geq 1$, define $\Omega_k = B_k \cap \Omega$, where $B_k$ is the ball of radius $k$ centered at a fixed vertex. Define the boundary function $F_k \colon \partial \Omega_k \to \mathbb{R}$ by
    $$
    F_k(x) =
    \begin{cases}
    1, & \text{if } x \in \partial \Omega \cap \partial \Omega_k, \\
    0, & \text{if } x \in \partial \Omega_k \setminus \partial \Omega.
    \end{cases}
    $$
    Since $u \ge 1$ on $\Omega^c$ and $u > 0$, we clearly have $F_k \le u$ on $\partial \Omega_k$. 
    
    Solving the Dirichlet problem
    $$
    \begin{cases}
    \Delta_p w = 0, & \text{in } \Omega_k, \\
    w = F_k, & \text{on } \partial \Omega_k,
    \end{cases}
    $$
    we obtain a unique solution $v_k$. By the comparison principle, the sequence $\{v_k\}$ is monotonically decreasing and satisfies $0 \leq v_k \leq 1$ as well as $v_k \le u$ in $\Omega_k$. 
    
    Define the pointwise limit
    $$
    v := \lim_{k \to \infty} v_k.
    $$
    Then $v$ is a $p$-harmonic function on $\Omega$ satisfying $0 \le v \le 1$ and $v \le u$. 
    
    Suppose, for the sake of contradiction, that there exists some $x_1 \in \Omega$ such that $v(x_1) = 1$. Since $v \le 1$ everywhere and $\Delta_p v(x_1) = 0$, the definition of the discrete $p$-Laplacian implies that $v(y) = 1$ for all neighbors $y \sim x_1$. Since $\Omega$ is connected, repeating this argument along paths in $\Omega$ yields $v \equiv 1$ on $\Omega$. In particular, this implies $v(x_0) = 1$, which contradicts the fact that $v(x_0) \le u(x_0) < 1$. Thus, we must have $v(x) < 1$ for all $x \in \Omega$, which implies that $\Omega^c$ is not $p$-thick at any point in $\Omega$.
    
    $(3) \Rightarrow (1)$: This implication follows from above construction directly.
\end{proof}

\begin{definition}\label{bc}
We say that $(V, \mu)$ satisfies the ball covering property \hyperref[bc]{$\text{(BC)}$} if, for every $\varepsilon > 0$, any ball $B(x, R)$ can be covered by $N = N(\varepsilon)$ balls of radius $\varepsilon R$.
\end{definition}
It is well known that \hyperref[VD]{$\text{(VD)}$} implies \hyperref[bc]{$\text{(BC)}$}.

\begin{definition}\rm
We say that the weighted graph $(V,\mu)$ satisfies the elliptic Harnack inequality \hyperref[hi]{$(H_p)$} for $p$-harmonic functions if there exist constants $t>1$ and $C_h>0$ such that for all $x \in V$, $R>0$, and every nonnegative $p$-harmonic function $u$ on $B(x,tR)$, the following estimate holds:
\begin{equation}\label{hi}
    \max_{B(x,R)} u \le C_h \min_{B(x,R)} u.
    \tag{$H_p$}
\end{equation}
\end{definition}

\begin{proposition}
    Assume that $(V,\mu)$ satisfies the condition \hyperref[p0]{$(p_0)$}. Then the volume doubling condition \hyperref[VD]{$\text{(VD)}$} and the Poincar\'e inequality \hyperref[PI]{$(P_p)$} together imply the elliptic Harnack inequality \hyperref[hi]{$(H_p)$}.
\end{proposition}

\begin{proof}
    Holopainen and Soardi \cite{HS97liou} established this result for simple graphs with uniformly bounded degrees. Their proof can be directly generalized to weighted graphs satisfying the \hyperref[p0]{$(p_0)$} condition.
\end{proof}

\begin{definition}\rm
We say that the weighted graph $(V,\mu)$ satisfies the Sobolev inequality \hyperref[SI]{$\text{(SI)}$} if there exist constants $\kappa>1$ and $C > 0$ such that for every ball $B=B(x,r)$ and all functions $u$ vanishing outside $B$, 
\begin{equation}
     \left(\frac{1}{\mu(B)} \sum_{x\in B} |u(x)|^{\kappa p} \mu(x)\right)^{\frac{1}{\kappa p}} \le C r \left(\frac{1}{\mu(B)} \sum_{x,y\in B}  |\nabla_{xy} u|^p\mu_{xy} \right)^{\frac{1}{p}}.
     \tag{$\text{SI}$}\label{SI}
\end{equation}
\end{definition}

It is well known that the volume doubling condition \hyperref[VD]{$\text{(VD)}$} and the Poincar\'e inequality \hyperref[PI]{$(P_p)$} together imply the Sobolev inequality \hyperref[SI]{$\text{(SI)}$}. This fundamental implication was established by Haj\l asz and Koskela \cite{HK95} in the general setting of metric measure spaces; see also \cite{HS97liou} for its counterpart on graphs.

\begin{theorem}[Wiener's criterion at $\infty$]\label{wcm}
Assume that $(V,\mu)$ satisfies \hyperref[VD]{$\text{(VD)}$} and \hyperref[PI]{$(P_p)$}. For an infinite set $A \subset V$, if $A$ is $p$-thick at $x_0$, then
\begin{equation}\label{wccap}
    \sum_{n=1}^{\infty} \left( \frac{\operatorname{cap}_{p}(A_n, B_{n+1})}{\operatorname{cap}_{p}(B_n, B_{n+1})} \right)^{\frac{1}{p-1}} = \infty,
\end{equation}
where $B_n = B(x_0, r_n)$ with $r_n = 2^n$, and $A_n = A \cap B_n$.

Moreover, if $(V,\mu)$ also satisfies \hyperref[p0]{$(\text{$p_0$})$}, then the converse also holds.
\end{theorem}
\begin{remark}
    \eqref{wccap} can be replaced with
     \begin{equation}\label{wcm2}
    \sum_{n=1}^{\infty} \left( \frac{r_n^p \operatorname{cap}_{p}(A_n, B_{n+1})}{\mu(B_{n})} \right)^{\frac{1}{p-1}} = \infty,
     \end{equation}
   if $(V,\mu)$ is not $p$-parabolic, \eqref{wccap} can also be replaced with
    \begin{equation}\label{wcvol}
    \sum_{n=1}^{\infty} \left( \frac{r_n^p\operatorname{cap}_{p}(A_n)}{\mu(B_{n})} \right)^{\frac{1}{p-1}} = \infty,
     \end{equation}
\end{remark}
\begin{remark}\rm
In the above and below, for a positive constant \(C\) whose exact value is unimportant but independent of the variables under consideration,  we write \(f\lesssim g\) (resp.\ \(f\gtrsim g\)) to mean that \(f\le Cg\) (resp.\ \(f\ge Cg\)). And the notation \(f\asymp g\) means that both \(f\lesssim g\) and \(f\gtrsim g\) hold.
\end{remark}

\begin{proof}[\textbf{Proof of sufficiency.}]
    Assume \eqref{wccap} holds. Let $u$ be a positive $p$-superharmonic function on $V$ such that $u \ge 1$ on $A$. 
    
    Let $u_n$ be the $p$-potential of the condenser $(A_n, B_{n+1})$. For any integer $n > 1$, define
    \[
    m_n = \min_{\partial B(x_0, 3 \cdot 2^{n-1})} u_n \quad \text{and} \quad M_n = \max_{\partial B(x_0, 3 \cdot 2^{n-1})} u_n.
    \]
    By Lemma \ref{cap_lvs}, for $\Gamma_\lambda := \{x \in V : u_n(x) > \lambda\}$, we have
    \[
    M_n^{p-1} \operatorname{cap}_{p}(\overline{\Gamma_{M_n}}, B_{n+1}) \ge \operatorname{cap}_{p}(A_n, B_{n+1}).
    \]
     Applying \eqref{hi} and \hyperref[bc]{$\text{(BC)}$} with $R=2^{n-2}/t$ yields
    \[
    m_n \le C_h^{N(6t)} M_n.
    \]
   Since $\overline{\Gamma_{M_n}} \subset B(x_0, 3\cdot 2^{n-1}+1)=:\frac{3}{2}B_n$, we have
    \[
     m_n \gtrsim  M_n \ge \left( \frac{\operatorname{cap}_{p}(A_n, B_{n+1})}{\operatorname{cap}_{p}(\frac{3}{2}B_n, B_{n+1})} \right)^{\frac{1}{p-1}}.
    \]
    From Lemma \ref{cap_balllob}, Lemma \ref{cap_ballupb} and (VD), we have the comparability
    \[
    \operatorname{cap}_{p}(\frac{3}{2}B_n, B_{n+1}) \asymp \frac{\mu(B_n)}{2^{pn}} \asymp \operatorname{cap}_{p}(B_n, B_{n+1}).
    \]
    It follows that
    \begin{equation}\label{lb}
        \min_{x \in \bar{B}_n} u_n \ge m_n \gtrsim \left( \frac{\operatorname{cap}_{p}(A_n, B_{n+1})}{\operatorname{cap}_{p}(B_n, B_{n+1})} \right)^{\frac{1}{p-1}}.
    \end{equation}

    Now, let $u$ be a $p$-superharmonic function on $V$ such that $u \ge 1$ on $A$. For $k \ge 1$, define
    \[
    a_k = \min_{x \in \bar{B}_k} u \quad \text{and} \quad b_k = \min_{x \in \bar{B}_k} u_k.
    \]
    Since $u - a_{k+1} \ge (1 - a_{k+1})u_k$ on $A_k \cup \partial B_{k+1}$, the comparison principle implies
    \[
    u(x) - a_{k+1} \ge (1 - a_{k+1})u_k(x) \quad \text{for } x \in B_{k+1}.
    \]
    Taking the minimum over $\bar{B}_{k-2}$, we obtain
    \[
    a_{k-2} - a_{k+1} \ge (1 - a_{k+1})b_{k-2},
    \]
    which implies
    \[
    1 - a_{k-2} \le (1 - a_{k+1})(1 - b_{k-2}).
    \]
    Let $c_k := 1 - a_k$. By iteration, we have
    \[
    c_n = 1 - a_n \le \prod_{i=0}^{\infty} (1 - b_{n+3i}).
    \]
    Noting that $c_n$ is non-decreasing in $n$ (as $a_n$ is non-increasing), we obtain
    \[
    c_n^3 \le \prod_{i=0}^{\infty} (1 - b_{n+i}).
    \]
    From \eqref{lb} and \eqref{wccap}, the series $\sum_{i=1}^{\infty} b_i$ diverges. Thus, the infinite product $\prod_{i=0}^{\infty} (1 - b_{n+i})$ vanishes for any $n>1$, which implies $c_n \le 0$, or $a_n \ge 1$. Consequently, $u \ge 1$ on $V$, and $A$ is $p$-thick at $x_0$.

    Noting that $\operatorname{cap}_{p}(A_n, B_{n+1})\ge \operatorname{cap}_{p}(A_n)$, Lemma \ref{cap_ballupb} also yields the sufficiency of \eqref{wcm2} and \eqref{wcvol}.
    \end{proof}

To prove the necessity part, we adapt the Kilpel\"{a}inen--Mal\'{y} strategy \cite{KM94} from finite boundary points in $\mathbb{R}^n$ to the point at infinity on a graph. The main additional difficulty is that the continuum integration-by-parts and scaling argument must be replaced by edge-based estimates for the measure term $\sigma=-\mu\Delta_p u$ and by a dyadic iteration over balls. The next two lemmas isolate exactly these new ingredients.

\begin{lemma}\label{lemc}
   Assume that $(V,\mu)$ satisfies \hyperref[VD]{$\text{(VD)}$} and
   \hyperref[PI]{$(P_p)$}. Let \(r\ge 2\), and let
   \(B_1=B(x_0,r)\) and \(B_2=B(x_0,\lfloor\frac{r}{2}\rfloor)\) be two
   concentric balls. Suppose that \(u\) is a non-negative \(p\)-supersolution
   on \(B_1\), and define \(\sigma(x) \coloneqq -\Delta_p u(x)\mu(x)\), so that
   \(\sigma(x)\ge 0\). For any \(a \in \mathbb{R}\) and any \(\gamma\)
   satisfying \(p-1 < \gamma < \frac{\kappa p (p-1)}{\kappa+p-1}\), let
   \(A \coloneqq \{x : u(x) > a\}\). If \(\lambda > 0\) satisfies
    \begin{align}\label{lemc-cond}
        \frac{\mu(B_1\cap A)}{\mu(B_1)} \le \frac{1}{2}\lambda^{\gamma}\left( \frac{1}{\mu(B_2)} \sum_{x \in B_2} (u(x)-a)_+^\gamma \mu(x) \right),
    \end{align}
    then there exists a constant $C$ depending on $p, \kappa$, and $\gamma$ such that
    \begin{align*}
       \bigg( \frac{\lambda^{\gamma}}{\mu(B_2)} \sum_{x \in B_2} (u(x)-a)_+^\gamma \mu(x) \bigg)^{1/\kappa} 
        \le \frac{C \lambda^{\gamma}}{\mu(B_1)}\sum_{x\in B_1}(u(x)-a)_{+}^{\gamma}\mu(x)
       + C\lambda^{p-1}r^p \frac{\sigma(B_1)}{\mu(B_1)}.
    \end{align*}
\end{lemma}

\begin{proof}
    Without loss of generality, we assume $a=0$. Consider the test function $\varphi(x) = \eta(x)^p \psi(u_{+}(x))$, where $u_{+} = \max\{u,0\}$. 
    Let $\eta$ be a cut-off function chosen such that $0 \le \eta \le 1$, suppose $n:=\lfloor\frac{r}{2}\rfloor$ and define
    \[
        \eta(x) \coloneqq \begin{cases} 
           1, & \text{if} \quad d(x, x_0) \leq n, \\
           0, & \text{if} \quad d(x, x_0) \geq 2n, \\
           \frac{2n-d(x, x_0)}{n}, & \text{otherwise.}
        \end{cases}
    \]
    Consequently, $\eta \equiv 1$ on $B_2$, $\text{supp}(\eta) \subset B(x_0,r-1)$, and $|\nabla_{xy} \eta| \le C/r$.

    Set $\alpha = \frac{\gamma}{p-1}$. Since $\gamma > p-1$, we have $\alpha > 1$. Define the function $\psi: [0, \infty) \to [0, 1)$ by
    \[
        \psi(t) = 1 - (1+t)^{1-\alpha}, \quad \text{with} \quad
        \psi'(t) = (\alpha-1)(1+t)^{-\alpha}.
    \]
    Applying the discrete Green formula in Lemma \ref{gf} on \(B_1\), we obtain
    \begin{align}\label{lemc-1}
        \sum_{x\in B_1} \varphi(x)\sigma(x)\mu(x) = \frac{1}{2}\sum_{x,y\in B_1} \mu_{xy}|\nabla_{xy}u|^{p-2}\nabla_{xy}u \nabla_{xy}\varphi.
    \end{align}
    We decompose the term $\nabla_{xy}\varphi$ as follows:
    \begin{align*}
        \varphi(y) - \varphi(x) = \eta(y)^p (\psi(u_{+}(y)) - \psi(u_{+}(x))) + \psi(u_{+}(x)) (\eta(y)^p - \eta(x)^p).
    \end{align*}
    Substituting this decomposition into \eqref{lemc-1} yields
    \begin{align}\label{lemc-2}
        2\sum_{x\in B_1} \varphi(x)\sigma(x)\mu(x) = & \sum_{x,y\in B_1} \mu_{xy}\eta(y)^p|\nabla_{xy}u_{+}|^{p-2}(\nabla_{xy}u_{+}) (\nabla_{xy}\psi(u_{+})) \nonumber \\ 
        &+ \sum_{x,y\in B_1} \mu_{xy}\psi(u_{+}(x))|\nabla_{xy}u_{+}|^{p-2}(\nabla_{xy}u_{+}) (\nabla_{xy}\eta^p ) \nonumber\\
        \eqqcolon & I_1 + I_2.
    \end{align}

    For each edge \(xy\), set
    \[
        U_{xy}\coloneqq \max\{u_{+}(x),u_{+}(y)\}
        \qquad\text{and}\qquad
        \eta_{\max}\coloneqq \max\{\eta(x), \eta(y)\}.
    \]
    Since \(\psi\) is increasing and \(\psi'(t)=(\alpha-1)(1+t)^{-\alpha}\) is decreasing, we have the edgewise lower bound
    \begin{align}\label{lemc-edge1}
        |\nabla_{xy}u_{+}|^{p-2}\nabla_{xy}u_{+}\nabla_{xy}\psi(u_{+})
        \ge c_\alpha (1+U_{xy})^{-\alpha}|\nabla_{xy}u_{+}|^{p},
    \end{align}
    where \(c_\alpha=\alpha-1\). Indeed, if \(u_{+}(y)\ge u_{+}(x)\), then
    \[
        \nabla_{xy}\psi(u_{+})
        =\int_{u_{+}(x)}^{u_{+}(y)}\psi'(t)\,dt
        \ge \psi'(u_{+}(y))(u_{+}(y)-u_{+}(x))
        =c_\alpha(1+U_{xy})^{-\alpha}\nabla_{xy}u_{+},
    \]
    and the case \(u_{+}(x)\ge u_{+}(y)\) is identical after exchanging \(x\) and \(y\).
    By symmetry of \(\mu_{xy}\), averaging the expressions obtained from \(I_1\) and from the same sum with \(x\) and \(y\) interchanged gives
    \begin{align}\label{lemc-i1}
        I_1
        &= \frac{1}{2}\sum_{x,y\in B_1} \mu_{xy}(\eta(x)^p+ \eta(y)^p)|\nabla_{xy}u_{+}|^{p-2}\nabla_{xy}u_{+}\nabla_{xy}\psi(u_{+}) \nonumber\\
        &\ge c\sum_{x,y\in B_1} \mu_{xy}\eta_{\max}^p(1+U_{xy})^{-\alpha}|\nabla_{xy}u_{+}|^{p}.
    \end{align}

    Next, to estimate $I_2$, we use \( |\eta(y)^p - \eta(x)^p| \le p \eta_{\max}^{p-1}|\eta(y) - \eta(x)| \), so that
    \begin{align*}
        |I_2| \le p\sum_{x,y\in B_1} \mu_{xy}\psi(u_{+}(x))\eta_{\max}^{p-1} |\nabla_{xy}u_{+}|^{p-1}|\nabla_{xy}\eta|.
    \end{align*}
    Applying Young's inequality with the weight \((1+U_{xy})^{-\alpha}\), we obtain
    \begin{align}\label{lemc-i2}
        |I_2| \le & \frac{c}{2}\sum_{x,y\in B_1} \mu_{xy} \eta_{\max}^{p}(1+U_{xy})^{-\alpha} |\nabla_{xy}u_{+}|^{p} \nonumber \\ 
        &+ C_p\sum_{x,y\in B_1} \mu_{xy}\psi(u_{+}(x))^p(1+U_{xy})^{\alpha(p-1)}|\nabla_{xy}\eta|^p.
    \end{align}

    Combining \eqref{lemc-i1}, \eqref{lemc-i2}, and \eqref{lemc-2} yields
    \begin{align}\label{lemc-sum}
        \sum_{x,y\in B_1} \mu_{xy}\eta_{\max}^p (1+U_{xy})^{-\alpha}|\nabla_{xy}u_{+}|^{p} \le & \, C \sum_{x,y\in B_1} \mu_{xy}\psi(u_{+}(x))^p(1+U_{xy})^{\alpha(p-1)}|\nabla_{xy}\eta|^p \nonumber \\ 
        &+ C\sum_{x\in B_1} \varphi(x)\sigma(x)\mu(x).
    \end{align}

    We now invoke the discrete Sobolev inequality. For the function $v \coloneqq \eta \phi(u_+)$, where $\phi(t) = \int^{t}_{0}\psi'(\tau)^{\frac{1}{p}}\, d\tau=C_{p,\gamma}((1+t)^{1-\frac{\alpha}{p}}-1)$, the Sobolev inequality states:
    \[
        \left(\frac{1}{\mu(B_1)} \sum_{x\in B_1} |v(x)|^{\kappa p} \mu(x)\right)^{1/\kappa p} \le C r \left(\frac{1}{\mu(B_1)} \sum_{x,y\in B_1} \mu_{xy} |\nabla_{xy} v|^p \right)^{1/p}.
    \]
    Using the identity $\nabla_{xy} (\eta \phi(u_+)) = \phi(u_{+}(x))\nabla_{xy} \eta + \eta(y) \nabla_{xy} \phi(u_{+})$, we obtain
    \begin{align}\label{lemc-si}
        \left(\frac{1}{\mu(B_1)} \sum_{x\in B_1} |\eta(x) \phi(u_{+}(x))|^{\kappa p} \mu(x)\right)^{1/\kappa} \le& \, C\frac{r^p}{\mu(B_1)}\Bigg[ \sum_{x,y\in B_1} \mu_{xy} \phi(u_{+}(x))^p|\nabla_{xy} \eta|^p \nonumber \\
        &+\sum_{x,y\in B_1} \mu_{xy} \eta(y)^p |\nabla_{xy} \phi(u_{+})|^p\Bigg].
    \end{align}

    The place where the continuum proof uses the chain rule is replaced here by the following two-point inequality along each edge:
    \begin{align}\label{lemc-edge2}
        |\nabla_{xy} \phi(u_{+})|^p
        &= \left| \int_{u_{+}(x)}^{u_{+}(y)} \psi'(t)^{1/p} \, dt\right|^p \nonumber\\
        &\le |\nabla_{xy}u_{+}|^{p-1}\left| \int_{u_{+}(x)}^{u_{+}(y)} \psi'(t) \, dt\right| \nonumber\\
        &= |\nabla_{xy}u_{+}|^{p-2}\nabla_{xy}u_{+}\nabla_{xy}\psi(u_{+}).
    \end{align}
    Hence
    \begin{align*}
        \sum_{x,y\in B_1} \mu_{xy} \eta(y)^p |\nabla_{xy} \phi(u_{+})|^p \le I_1.
    \end{align*}
    Set
    \begin{align*}
        J_1 &:= \sum_{x,y\in B_1} \mu_{xy}
        \psi(u_{+}(x))^p(1+U_{xy})^{\alpha(p-1)} |\nabla_{xy}\eta|^p, \\
        J_2 &:= \sum_{x,y\in B_1} \mu_{xy}
        \phi(u_{+}(x))^p |\nabla_{xy}\eta|^p.
    \end{align*}
    Combining this estimate with \eqref{lemc-si} and \eqref{lemc-sum}, we obtain
    \begin{align*}
        \left(\frac{1}{\mu(B_1)}
        \sum_{x\in B_1} |\eta(x)\phi(u_{+}(x))|^{\kappa p} \mu(x)\right)^{1/\kappa}
        &\le \frac{Cr^p}{\mu(B_1)}(J_1+J_2)
        \nonumber \\
        &\quad + C\frac{r^p}{\mu(B_1)}
        \sum_{x\in B_1} \eta(x)^p \psi(u_{+}(x))\sigma(x)\mu(x).
    \end{align*}
      
    Then, using the volume doubling property \hyperref[VD]{$\text{(VD)}$}
    and the fact that $\eta \equiv 1$ on $B_2$, we have
    \begin{align}\label{lemc-3}
        \Bigg(\frac{1}{\mu(B_2)} \sum_{x\in B_2}
        ((1+u_+(x))^{1-\alpha}-1)^{\kappa p} \mu(x)\Bigg)^{1/\kappa}
        \nonumber \\
        &\le C \Bigg(\frac{1}{\mu(B_1)}
        \sum_{x\in B_1} |\eta(x)\phi(u_{+}(x))|^{\kappa p} \mu(x)\Bigg)^{1/\kappa}
        \nonumber \\
        &\le C\frac{r^p}{\mu(B_1)}(J_1+J_2)
        \nonumber \\
        &\quad + C\frac{r^p}{\mu(B_1)}
        \sum_{x\in B_1} \eta(x)^p \psi(u_{+}(x))\sigma(x)\mu(x).
    \end{align}

    Since \(\psi\) is increasing and bounded by \(1\), it follows that
    \begin{align*}
       \psi(u_{+}(x))^p(1+U_{xy})^{\alpha(p-1)}
       \le C(1+U_{xy})^{\gamma}
       \le C(1+u_{+}(x)^{\gamma}+u_{+}(y)^{\gamma}).
    \end{align*}

    Then, from $\phi(t)^p\le C t^{\gamma}$ and
    $\mu(x)=\sum_{y\in V}\mu_{xy}$, we have
    \begin{align*}
        r^p(J_1+J_2)
        &\le C\sum_{x\in A\cap B_1}\sum_{y\in B_1} \mu_{xy}
        \bigl(1+2u_{+}(x)^{\gamma}+u_{+}(y)^{\gamma}\bigr) \\
        &\le C \sum_{x\in A\cap B_1}\mu(x)
        + C\sum_{x\in A\cap B_1}u_{+}(x)^{\gamma}\mu(x)
        + C\sum_{y\in B_1}u_{+}(y)^{\gamma}\mu(y) \\
        &\le C \sum_{x\in A\cap B_1}\mu(x)
        + C\sum_{x\in B_1}u_{+}(x)^{\gamma}\mu(x).
    \end{align*}

    Substituting this into \eqref{lemc-3}, we have
    \begin{align*}
        \Bigg(\frac{1}{\mu(B_2)} \sum_{x\in B_2}
        ((1+u_+(x))^{1-\alpha}-1)^{\kappa p} \mu(x)\Bigg)^{1/\kappa}
        &\le C\frac{r^p\sigma(B_1)}{\mu(B_1)}
        + C \frac{\mu(B_1\cap A)}{\mu(B_1)}
        \nonumber \\
        &\quad + \frac{C}{\mu(B_1)}
        \sum_{x\in B_1}u_{+}(x)^{\gamma}\mu(x).
    \end{align*}

    We split the set $A$ into two subsets: $A_1=\{x:u(x)>1\}$ and $A_2=\{x:0<u(x)\le 1\}$. Then, from \(\kappa>1\) and \hyperref[VD]{$\text{(VD)}$},
    \begin{align*}
        \Bigg( \frac{1}{\mu(B_2)} \sum_{x \in B_2} u(x)_+^\gamma \mu(x)
        \Bigg)^{1/\kappa}
        &\le C\left(\frac{\mu(B_2\cap A)}{\mu(B_2)}\right)^{1/\kappa}
        \nonumber \\
        &\quad + C\left( \frac{1}{\mu(B_2)}
        \sum_{x \in B_2\cap A_1} u(x)_+^\gamma \mu(x) \right)^{1/\kappa} \\
        &\le C\left(\frac{\mu(B_1\cap A)}{\mu(B_1)}\right)^{1/\kappa}
        \nonumber \\
        &\quad + C \Bigg(\frac{1}{\mu(B_2)} \sum_{x\in B_2}
        ((1+u_+(x))^{1-\alpha}-1)^{\kappa p} \mu(x)\Bigg)^{1/\kappa} \\
        &\le C\left(\frac{\mu(B_1\cap A)}{\mu(B_1)}\right)^{1/\kappa}
        + C \frac{r^p\sigma(B_1)}{\mu(B_1)} \nonumber \\
        &\quad + \frac{C}{\mu(B_1)}\sum_{x\in B_1}u_{+}(x)^{\gamma}\mu(x).
    \end{align*}
    
    Consider the scaled function $\hat{u}=\lambda u$, where $\lambda>0$ is chosen such that
    \[
    \frac{\mu(B_1\cap A)}{\mu(B_1)} \le \frac{1}{2}\lambda^{\gamma}\left( \frac{1}{\mu(B_2)} \sum_{x \in B_2} u(x)_+^\gamma \mu(x) \right).
    \]
    Noting that $\hat{\sigma}=\lambda^{p-1}\sigma$, we derive
    \begin{align*}
        \left( \frac{\lambda^{\gamma}}{\mu(B_2)}
        \sum_{x \in B_2} u(x)_+^\gamma \mu(x) \right)^{1/\kappa}
        \le \frac{C \lambda^{\gamma}}{\mu(B_1)}
        \sum_{x\in B_1}u_{+}(x)^{\gamma}\mu(x)
        + C\lambda^{p-1}r^p \frac{\sigma(B_1)}{\mu(B_1)},
    \end{align*}
    which completes the proof.
\end{proof}

\begin{lemma}\label{suph_upb}
        Let $(V,\mu)$ be a graph satisfying \hyperref[VD]{$\text{(VD)}$}
        and \hyperref[PI]{$(P_p)$}. Let \(B=B(x_0,r)\) and
        \(2B=B(x_0,2r)\). Suppose \(u\) is a non-negative supersolution on
        \(2B\), and define
        \[
            \sigma(x):= -\Delta_p u(x)\mu(x).
        \]
        Then, for every \(\gamma>p-1\), there exists a constant \(C\),
        depending only on \(p\), \(\gamma\), \(\kappa\), \(C_1\), and \(C_2\),
        such that
        \begin{align*}
          u(x_0)
          \le C \left( \frac{1}{\mu(B)} \sum_{x \in B} u(x)^\gamma \mu(x)
          \right)^{1/\gamma} + C\sum^{n}_{i=1}
          \left(\frac{r_i^p\sigma(B_i)}{\mu(B_i)}\right)^{\frac{1}{p-1}},
      \end{align*}
      where \(B_{i}=B(x_0,\lfloor2^{2-i}r\rfloor)\) for \(i=1,\dots, n\),
      and \(B_n=B(x_0,1)=\{x_0\}\).
\end{lemma}

\begin{proof}
    By Holder's inequality, for any \(\gamma_1<\gamma_2\),
    \begin{align*}
     \left( \frac{1}{\mu(B)} \sum_{x \in B} u(x)^{\gamma_1} \mu(x) \right)^{1/\gamma_1}\le \left( \frac{1}{\mu(B)} \sum_{x \in B} u(x)^{\gamma_2} \mu(x) \right)^{1/\gamma_2},
    \end{align*}
    we assume \(p-1 < \gamma < \frac{\kappa p (p-1)}{\kappa+p-1}\).

    Let \(n=min\{k\in \N:r< 2^{k-1}\}\). For \(i=1,\dots, n\), define
    \(B_{i}=B(x_0,r_i)\), where \(r_i=\lfloor2^{2-i}r\rfloor\) and \(r_{n}=0\).
    Thus \(\{B_i\}_{i=1}^n\) is a dyadic chain of balls shrinking from radius
    comparable to \(r\) down to the single vertex \(x_0\). Fix some constant
    \(\delta>0\) to be determined later, let \(a_0=0\), and for \(0\le i\le n\),
    define
    \begin{align*}
        a_{i+1}=a_{i}+\delta^{-1}\left( \frac{1}{\mu(B_{i+1})} \sum_{x \in B_{i+1}} (u(x)-a_i)_+^{\gamma} \mu(x) \right)^{1/\gamma},
    \end{align*}
    clearly, \(a_{i+1}>a_{i}\). Hence for \(i\ge 1\),
    \begin{align}\label{theoup-vlm}
        \mu(B_i\cap\{x:u(x)>a_i\})&\le (a_i-a_{i-1})^{-\gamma}\sum_{\substack{x\in B_i\\u(x)>a_i}}(u(x)-a_{i-1})^{\gamma} \mu(x) \nonumber\\
        &\le (a_i-a_{i-1})^{-\gamma}\sum_{x\in B_i}(u(x)-a_{i-1})_{+}^{\gamma} \mu(x) \nonumber\\
        &=\delta^{\gamma}\mu(B_{i}).
    \end{align}

    Now letting \(\lambda_{i-1}=2^{1/\gamma}(a_{i}-a_{i-1})^{-1}\), we have
    \begin{align*}
        \frac{\mu(B_i\cap\{x:u(x)>a_i\})}{\mu(B_{i})}=\delta^{\gamma}=\frac{1}{2}\lambda_i^{\gamma}\left( \frac{1}{\mu(B_{i+1})} \sum_{x \in B_{i+1}} (u(x)-a_i)_+^{\gamma} \mu(x) \right),
    \end{align*}
    which satisfies \eqref{lemc-cond}.
    Applying Lemma \ref{lemc}, we obtain
    \begin{align*}
        \bigg( \frac{\lambda_i^{\gamma}}{\mu(B_{i+1})} \sum_{x \in B_{i+1}} (u(x)-a_i)_+^\gamma \mu(x) \bigg)^{1/\kappa} 
        \le \frac{C \lambda_i^{\gamma}}{\mu(B_i)}\sum_{x\in B_i}(u(x)-a_i)_{+}^{\gamma}\mu(x)
       + C\lambda_i^{p-1}r_i^p \frac{\sigma(B_i)}{\mu(B_i)},
    \end{align*}
  which is
  \begin{align*}
      (2\delta^{\gamma})^{1/\kappa}\le C\delta^{\gamma}\frac{\lambda_i^{\gamma}}{\lambda_{i-1}^{\gamma}}+C\lambda_i^{p-1}r_i^p \frac{\sigma(B_i)}{\mu(B_i)}.
  \end{align*}

    From (VD), we have
    \begin{align*}
        \frac{\lambda_i^{\gamma}}{\lambda_{i-1}^{\gamma}}=\frac{(a_{i}-a_{i-1})^\gamma}{(a_{i+1}-a_{i})^\gamma}=\frac{\mu(B_{i+1})\sum_{x \in B_{i}} (u(x)-a_{i-1})_+^{\gamma} \mu(x)}{\mu(B_{i})\sum_{x \in B_{i+1}} (u(x)-a_i)_+^{\gamma} \mu(x)}\le C.
    \end{align*}

    It follows that
    \begin{align*}
      (2\delta^{\gamma})^{1/\kappa}\le C\delta^{\gamma}+C(a_{i+1}-a_i)^{1-p}r_i^p \frac{\sigma(B_i)}{\mu(B_i)}.
  \end{align*}

  Now by fixing some \(\delta<2^{-1/\gamma} C^{-\frac{\kappa}{\gamma(\kappa-1)}}\), we obtain
  \begin{align*}
      a_{i+1}-a_i \le C\left(r_i^p \frac{\sigma(B_i)}{\mu(B_i)}\right)^{1/(p-1)}.
  \end{align*}
 Hence
 \begin{align*}
      a_n=a_1+\sum_{i=1}^{n-1}(a_{i+1}-a_{i})\le a_1+C\sum_{i=1}^{n-1}\left(r_i^p \frac{\sigma(B_i)}{\mu(B_i)}\right)^{1/(p-1)}.
 \end{align*}
 Since \(\delta<1\) and \(B_n=\{x_0\}\), taking \(i=n\) in \eqref{theoup-vlm} gives \(u(x_0)\le a_n\). This completes the proof, because
\[
a_1=\delta^{-1} \left( \frac{1}{\mu(B)} \sum_{x \in B} u(x)_+^\gamma \mu(x) \right)^{1/\gamma}.
\]
\end{proof}

Now we are ready to prove the necessity of the theorem \ref{wcm}.

    \begin{proof}[\textbf{Proof of necessity.}]
    Now let us assume 
    \begin{equation*}
    \sum_{n=1}^{\infty} \left( \frac{\operatorname{cap}_{p}(A_n, B_{n+1})}{\operatorname{cap}_{p}(B_n, B_{n+1})} \right)^{\frac{1}{p-1}} < \infty,
\end{equation*}
    then there exists some integer \(K>0\) such that
    \begin{align*}
        \sum_{n=K}^{\infty} \left( \frac{\operatorname{cap}_{p}(A_n, B_{n+1})}{\operatorname{cap}_{p}(B_n, B_{n+1})} \right)^{\frac{1}{p-1}} < \frac{1}{2C},
    \end{align*}
    where \(C=(C_1 C_2)^{1/(p-1)} C_3\).

    Define \(A'=A\setminus B_{K-1}\) and \(A'_n=A'\cap B_n\), we have
    \begin{align}\label{theo_wn1}
        \sum_{n=K}^{\infty} \left( \frac{\operatorname{cap}_{p}(A'_n, B_{n+1})}{\operatorname{cap}_{p}(B_n, B_{n+1})} \right)^{\frac{1}{p-1}} < \frac{1}{2C},
    \end{align}

    Let \(v_n\) be the $p$-potential of the condenser $(A'_n, B_{n+1})$. By Lemma \ref{suph_upb}, we obtain
    \begin{align*}
      v_n(x_0)\le C \left( \frac{1}{\mu(B_n)} \sum_{x \in B_n} v_n(x)^\gamma \mu(x) \right)^{1/\gamma}+C\sum^{n+1}_{i=1}\left(\frac{r_i^p\sigma_n(B_i)}{\mu(B_i)}\right)^{\frac{1}{p-1}},
    \end{align*}
    where \(\gamma>p-1\) and \(\sigma_n=-\mu\Delta v_n\).

    By Lemma \ref{cap_ballupb} and Lemma \ref{cap_pt}, we have
    \[
      \frac{r_n^p\sigma_n(B_n)}{\mu(B_n)} \le C_1 C_2  \frac{\operatorname{cap}_{p}(A'_n, B_{n+1})}{\operatorname{cap}_{p}(B_n, B_{n+1})}.
    \]
    Hence, noting that \(\sigma_n=0\) on \(B_i\) for \(i< K\), we have
    \begin{align}\label{theo_wn2}
         v_n(x_0)&\le C_3 \left( \frac{1}{\mu(B_n)} \sum_{x \in B_n} v_n(x)^\gamma \mu(x) \right)^{1/\gamma}+C\sum^{n+1}_{i=K}\left(\frac{\operatorname{cap}_{p}(A'_i, B_{i+1})}{\operatorname{cap}_{p}(B_i, B_{i+1})}\right)^{\frac{1}{p-1}}
         \nonumber \\
         &\le C_3 \left( \frac{1}{\mu(B_n)} \sum_{x \in B_n} v_n(x)^\gamma \mu(x) \right)^{1/\gamma}+\frac{1}{2}.
    \end{align}

     Noting \(v_k\le 1\) and \(\gamma>p-1\), using \hyperref[PI]{$(P_p)$}, we obtain
    \begin{align*}
      \frac{1}{\mu(B_n)} \sum_{x \in B_n} v_n(x)^\gamma \mu(x) &\le \frac{1}{\mu(B_n)} \sum_{x \in B_n} v_n(x) \mu(x) 
      \\
      &\lesssim  \left(\frac{r_n^p}{\mu(B_n)} \sum_{x,y \in B_{n+1}} |\nabla_{xy}v_n|^p \mu_{xy} \right)^{\frac{1}{p}}
      \\
      &=\left(\frac{r_n^p\operatorname{cap}_{p}(A'_n, B_{n+1})}{\mu(B_n)}\right)^{\frac{1}{p}}.
    \end{align*}

    Noting that by \hyperref[VD]{$\text{(VD)}$}, Lemma \ref{cap_ballupb} and \eqref{theo_wn1},  we have
    \begin{align*}
       \frac{r_n^p\operatorname{cap}_{p}(A'_n, B_{n+1})}{\mu(B_n)}\to 0, \quad\text{when $n\to \infty$},
    \end{align*}
    it follows that
    \begin{align*}
        \frac{1}{\mu(B_n)} \sum_{x \in B_n} v_n(x)^\gamma \mu(x) \to 0, \quad\text{when $n\to \infty$}.
    \end{align*}

    Then from \eqref{theo_wn2}, we obtain
    \begin{align*}
        \lim_{n\to \infty} v_n(x_0) \le \frac{1}{2}.
    \end{align*}

    By the comparison principle, the sequence \(v_n\) is increasing in \(n\). Let \(v:=\lim\limits_{n\to \infty}v_n\). Clearly, \(v\equiv1\) on \(A'\) and \(v\) is $p$-harmonic outside \(A'\). Hence, \(A'\) is not thick at \(x_0\).

    Since \(A\setminus A'\) is finite, \(A\) is also not thick at \(x_0\).

    For the proof of \eqref{wcvol}, we let \(u_n^k\) be the $p$-potential of $(A'_n, B_k)$ for $k>n$, noting $\{u_n^k\}$ is increasing by $k$, we define
    \[
    u_n:=\lim_{k \to\infty} u_n^k.
    \]

    Then, noting that \(u_n\) is $p$-superharmonic on \(V\), Lemma \ref{suph_upb} gives, for any $k>n$,
    \begin{align*}
         u_n(x_0)&\le C \left( \frac{1}{\mu(B_k)} \sum_{x \in B_k} u_n(x)^\gamma \mu(x) \right)^{1/\gamma}+C\sum^{k+1}_{i=K}\left(\frac{r_i^p\sigma_n(B_i)}{\mu(B_i)}\right)^{\frac{1}{p-1}}
         \nonumber\\
         & \le C \left( \frac{1}{\mu(B_k)} \sum_{x \in B_k} u_n(x)^\gamma \mu(x) \right)^{1/\gamma}+C\sum_{i=K}^{\infty} \left( \frac{r_i^p\operatorname{cap}_{p}(A'_i)}{\mu(B_{i})} \right)^{\frac{1}{p-1}},
    \end{align*}
    and  similar to above, we also have
    \begin{align*}
      \frac{1}{\mu(B_k)} \sum_{x \in B_k} u_n(x)^\gamma \mu(x) 
      &\lesssim  \left(\frac{r_k^p}{\mu(B_k)} \sum_{x,y \in B_{k+1}} |\nabla_{xy}u_n|^p \mu_{xy} \right)^{\frac{1}{p}}+\min_{x\in B_k}u_n(x)
      \\
      &\le (\frac{r_n^p\operatorname{cap}_{p}(A'_n)}{\mu(B_{n})})^{\frac{1}{p}}+
      \min_{x\in B_k}u_n(x),
    \end{align*}
    Since \((V,\mu)\) is not \(p\)-parabolic, \(u_n\) is not constant and is the minimal admissible function of \(A'_n\) which is $p$-harmonic on \((A'_n)^c\),  it follows that \(\min\limits_{x\in V}u_n(x)=0\). Then, letting \(k \to \infty\), we have
     \begin{align*}
        u_n(x_0) \le\frac{1}{2}.
     \end{align*}
    The same argument completes the proof of the necessity for \eqref{wcvol}.

\end{proof}
\section{Criterion for \texorpdfstring{$D_p$}{Dp}-Massive Sets}\label{cr_dpm}

We next turn to the finite-energy case. The
goal of this section is to prove the characterization of $D_p$-massive sets in
terms of non-parabolic subsets of finite relative capacity.

\begin{theorem}\label{dpmcr}
Let \( G = (V, E) \) be an infinite, connected, and locally finite graph. A subset \( \Omega \subset V \) is \( D_p \)-massive if and only if there exists a non-parabolic subset \( \Omega_1 \subset \Omega \) such that
\[
\operatorname{cap}_p(\Omega_1, \Omega) < \infty.
\]
\end{theorem}

\begin{proof}
Let \( \Omega \subset V \) be a \( D_p \)-massive set, and let \( u \) be an admissible \( p \)-harmonic function associated with \( \Omega \). Define the level set
\[
\Omega_1 := \left\{ x \in \Omega : u(x) < \frac{1}{2} \right\}.
\]
We claim that \( \Omega_1 \) is not $p$-parabolic and satisfies \( \operatorname{cap}_p(\Omega_1, \Omega) < \infty \).

For every \(x\in \Omega_1\), we have
\[
\Delta_{\Omega_1,p} u(x)\le \Delta_p u(x)=0,
\]
because \(u\) is \(p\)-harmonic on \(\Omega\), and every edge \(xy\in \partial_e\Omega_1\) with \(x\in \Omega_1\) satisfies \(u(y)\ge \frac12>u(x)\). Thus \(u\) is a positive \(p\)-superharmonic function on \(\Omega_1\). Since \(u\) is nonconstant on \(\Omega_1\), Proposition \ref{prop:parabolic_equiv} and the remark following it imply that \(\Omega_1\) is not $p$-parabolic.

Next define
\[
\eta:=\min\{1,\, 2(1-u)\}.
\]
Then \(\eta=0 \) on \(\Omega^c\) and \(\eta\equiv 1\) on \(\Omega_1\). Hence \(\eta\) is admissible for \(\operatorname{cap}_p(\Omega_1,\Omega)\). Moreover, we have
\[
D_p(\eta)\le 2^p D_p(u)<\infty.
\]
Therefore,
\[
\operatorname{cap}_p(\Omega_1, \Omega)\le D_p(\eta)<\infty.
\]

Now we assume \( \Omega_1 \subset \Omega \) is a non-parabolic subset such that
\[
\operatorname{cap}_p( \Omega_1,  \Omega) < \infty,
\]
and let \( \{B_k\} \) be an exhaustion of \( V \) by finite subsets. For each integer \( k \geq 2 \), define
\[
\Omega_k := \Omega_1 \cap B_{k+1}^c.
\]

For integers \( k > n+1 \), let \( v_n^k \) denote the potential function of the capacitor \( (\Omega_n, \Omega, B_k) \). Then,
\[
D_p(v_n^k;B_k) = \operatorname{cap}_p( \Omega_n,  \Omega, B_k).
\]

By a diagonal argument, the sequence \( \{v_n^k\} \) has a pointwise limit function \( v_n \). Since
\[
\operatorname{cap}_p(\Omega_n, \Omega, B_k) \leq \operatorname{cap}_p( \Omega_n,  \Omega, V),
\]
applying Fatou's lemma yields
\[
D_p(v_n) \leq \operatorname{cap}_p(\Omega_n, \Omega, V).
\]
Moreover, since \( v_n \) is admissible for the capacitor \( ( \Omega_n,  \Omega) \), we have equality:
\[
D_p(v_n) = \operatorname{cap}_p( \Omega_n, \Omega).
\]

By the maximum principle, for sufficiently large \( k \), we have \( v_n^k \leq v_{n+1}^k \). Taking limits, it follows that the sequence \( \{v_n\} \) is increasing, and we define
\[
v := \lim_{n \to \infty} v_n.
\]
Clearly, \(v\) is p-harmonic on \(\Omega\) with \(v\equiv0\) on \(\partial \Omega\), and by Fatou's lemma, 
\begin{align*}
    D_p(v) \le \liminf_{n \to \infty}D_p(v_n) = \liminf_{n\to\infty}\operatorname{cap}_p(\Omega_n, \Omega) \le \operatorname{cap}_p(\Omega_1, \Omega) < \infty.
\end{align*}
We will finish the proof by showing \(v\) is not constant.

Now, choose a sufficiently large integer \( n_0 \) such that both \( B_{n_0} \cap \Omega \) and \( B_{n_0} \cap \Omega_1 \) are non-empty. Then, for all \(n> n_0\),
\[
\operatorname{cap}_p(\Omega_n, \Omega)=\operatorname{C}_p(\Omega_n, \partial \Omega) \ge \operatorname{C}_p(\Omega_n, B_{n_0} \cap \partial\Omega)= \operatorname{C}_p(B_{n_0} \cap \partial\Omega, \Omega_n).
\]
Noting \(\Omega_n := \Omega_1 \cap B_{n+1}^c\), we have
\begin{align*}
    \operatorname{C}_p(B_{n_0} \cap \partial\Omega, \Omega_n)
\geq \operatorname{C}_p(B_{n_0} \cap \partial\Omega, \Omega_n; B_{n_0} \cup \Omega_1)=\operatorname{C}_p(B_{n_0} \cap \partial\Omega, B_n^c;B_{n_0} \cup \Omega_1).
\end{align*}

Since \( \Omega_1 \) is non-parabolic and \(B_{n_0}\) is finite, the set \( B_{n_0} \cup \Omega_1 \) is also non-parabolic. Therefore, for all \( n \geq n_0 \), we have
\[
\operatorname{C}_p(B_{n_0} \cap \partial\Omega, \infty;B_{n_0} \cup \Omega_1)=:C> 0.
\]
It follows that
\begin{align*}
    \operatorname{C}_p(B_{n_0} \cap \partial\Omega, \Omega_n)\ge C.
\end{align*}

Hence, by Lemma \ref{lem_eng} with \(t=0\), we conclude that
\[
 D_p(v_n) = \operatorname{cap}_p(\Omega_n, \Omega)=\sum_{x \in \Omega}\sum_{y\in \partial \Omega}|v_n(x)-v_n(y)|^{p-1}\mu_{xy}\ge C.
\]

Noting that \(v_n=0\) on \(\partial \Omega\) and \(\{v_n\} \) is increasing, by the monotone convergence theorem, we have
\begin{align*}
    \sum_{x \in \Omega}\sum_{y\in \partial \Omega}v(x)^{p-1}\mu_{xy}=\lim_{n \to \infty}\sum_{x \in \Omega}\sum_{y\in \partial \Omega}v_n(x)^{p-1}\mu_{xy} \ge C,
\end{align*}
Hence \(v\) is not constant. We complete the proof.
\end{proof}

\section{Examples}\label{example}

We conclude with examples illustrating the two criteria proved above and their
connection with familiar lattice situations. First, we discuss some graphs
that satisfy the volume doubling condition \hyperref[VD]{$\text{(VD)}$} and the
weak $(1,p)$-Poincaré inequality \hyperref[PI]{$(P_p)$}. For instance, Cayley
graphs of discrete finitely generated groups of polynomial growth satisfy
\hyperref[VD]{$\text{(VD)}$} and \hyperref[PI]{$(P_1)$} (see \cite{CS93},
\cite{CS95}).

In this section, we study massive sets on $\mathbb{Z}^d$ equipped with the standard counting measure $\mu$ and edge weights $\frac{1}{2d}$. Before doing so, we first estimate the $p$-capacity of cylinders. 

\begin{lemma}\label{lem_capset}
    For any $d > p > 1$, let $E \subset \mathbb{Z}^{d}$. Then
    \begin{align}\label{cap_set}
        \operatorname{cap}_p(E) \gtrsim \mu(E)^{\frac{d-p}{d}}.
    \end{align}
\end{lemma}

\begin{proof}
    For $1 < p < d$, it is well-known that the capacity lower bound \eqref{cap_set} is equivalent to the global Sobolev inequality 
    \begin{align}\label{gsob}
       \left( \sum_{x\in \mathbb{Z}^d} |f|^{\frac{dp}{d-p}}\right)^{\frac{d-p}{dp}} \le C\left( \sum_{x,y\in \mathbb{Z}^d} |\nabla_{xy} f|^{p}\right)^{\frac{1}{p}} \quad \text{for all $f\in \ell_0(\mathbb{Z}^d)$}
    \end{align}
    (see, for instance, \cite[Section 10.1]{BCLS95} and \cite[Section 5.1]{CK04}).

    Furthermore, as discussed in \cite[Section 2.6]{CK04}, $\mathbb{Z}^{d}$ satisfies \eqref{gsob} because it admits the volume doubling property \hyperref[VD]{$\text{(VD)}$}, the \((p,p)\) Poincaré inequality and the volume lower bound condition \eqref{lvq}. Here, graph $(V,\mu)$ is said to satisfy the volume lower bound condition \eqref{lvq} if there exist constants $C, Q > 0$ such that
    \begin{align}\label{lvq}
        V(B(x,r)) \ge C r^{Q} \quad \text{for all $x\in V$ and $r>0$}.\tag{$LV_Q$}
    \end{align}
    And a direct calculation shows that the $(1,1)$-Poincaré inequality \hyperref[PI]{$(P_1)$} implies the $(p,p)$-Poincaré inequality:
    \begin{equation*}
    	\left(\frac{1}{\mu(B)}\sum_{x\in B} |f(x)-f_B|^p\mu(x) \right)^{\frac{1}{p}} \le C r \left(\frac{1}{\mu(2B)}\sum_{x,y\in 2B}|f(y)-f(x)|^p\mu_{xy}\right)^{\frac{1}{p}}.
    \end{equation*}

    Since $\mathbb{Z}^d$ has polynomial volume growth of dimension $d$, it satisfies \eqref{lvq}, which yields the global Sobolev inequality \eqref{gsob} and, consequently, \eqref{cap_set}.
\end{proof}

\begin{lemma}\label{cap_c}
    Assume $d > p+1 > 2$ and $h \ge r$. Define the cylinder $C = \{ (x_1, x') \in \mathbb{Z} \times \mathbb{Z}^{d-1} : 0 \le x_1 \le h, \|x'\| \le r \}$, then
    \begin{equation}
        \operatorname{cap}_p(C) \asymp h r^{d-p-1},
    \end{equation}
     where $\|x'\|=\sqrt{x_2^2+\dots x_{d}^2}$ denotes the standard Euclidean norm in $\mathbb{Z}^{d-1}$.
\end{lemma}
\begin{proof}
    For the upper bound, we define a product cut-off function $\psi(x_1, x') = \phi(x_1)\eta(x')$. Let $C_1 = [0, h] \cap \mathbb{Z}$ and $C' = \{x' \in \mathbb{Z}^{d-1} : \|x'\| \le r\}$. Let $d_1(x_1) = \operatorname{dist}(x_1, C_1)$ and $d'(x') = \operatorname{dist}(x', C')$ denote the Euclidean distances on $\mathbb{Z}$ and $\mathbb{Z}^{d-1}$, respectively. We define:
    \[
    \eta(x') := \begin{cases} 
        1, & \text{if } d'(x') = 0, \\ 
        1 - \frac{d'(x')}{r}, & \text{if } 0 < d'(x') \le r, \\
        0, & \text{if } d'(x') > r,
    \end{cases}
    \quad \text{and} \quad
    \phi(x_1) := \begin{cases} 
        1, & \text{if } d_1(x_1) = 0, \\ 
        1 - \frac{d_1(x_1)}{r}, & \text{if } 0 < d_1(x_1) \le r, \\
        0, & \text{if } d_1(x_1) > r.
    \end{cases}
    \]
    Given that $h \ge r$, the $p$-capacity is bounded by the $p$-energy of $\psi$:
    $$\operatorname{cap}_p(C) \le \sum_{x \sim y} |\psi(x) - \psi(y)|^p.$$
    
    The sum can be partitioned into three regions: for $d_1 = 0$ and $0 < d' \le r$, the energy of $\psi$ is less than $h \cdot (2r)^{d-1} \cdot r^{-p} \asymp h r^{d-p-1}$. Next, the region where $d' = 0$ and $0 < d_1 \le r$ contributes $r^{d-1} \cdot r \cdot r^{-p} = r^{d-p}$. Finally, the region where $0 < d_1, d' \le r$ also contributes an energy of order $r^{d-p}$.

  Combining these contributions and noting that $h \ge r$:
  $$\operatorname{cap}_p(C) \lesssim r^{d-p} + h r^{d-p-1} \lesssim h r^{d-p-1}.$$
  This establishes the upper bound.

Moreover, observe that for any $v \in \mathcal{A}(C)$, we have
$$
\sum_{x,y\in \mathbb{Z}^d} |\nabla_{xy}v|^p \ge \sum_{n = 0}^{h} \left( \sum_{x' \sim y' \text{ in } \mathbb{Z}^{d-1}} |v(n, y') - v(n, x')|^p \right).
$$
Noting $d > p+1$, using Lemma \ref{lem_capset}, we have  $\operatorname{cap}_p^{(d-1)}(B_r) \gtrsim r^{d-1-p}$. Consequently, we obtain
\begin{equation}\label{cl_lo}
    \operatorname{cap}_p(C) \ge h \operatorname{cap}_p^{(d-1)}(B_{f(n)}) \gtrsim hr^{d-1-p},
\end{equation}
which completes the proof.
\end{proof}

Next, we discuss a widely studied example, the Thorn set $\mathcal{T}$ (see, \cite{IM60, gri88} for p=2 in graphs and manifolds setting and \cite{BC15} for similar example in $\alpha$-stable random walk).

\begin{example}\label{thorn_thick}
Let $d > p + 1$ and $f: \mathbb{N} \to (0, \infty)$ be a monotonically increasing sequence. Define the Thorn set $\mathcal{T}$ as
\begin{equation*}
    \mathcal{T} = \{(x_1, x') \in \mathbb{Z}^d : x_1 \ge 0, \|x'\| \le f(x_1)\}.
\end{equation*}
Then $\mathcal{T}$ is $p$-thick at $\infty$ if and only if
\begin{equation}\label{eq:thorn_series}
    \sum_{n=1}^{\infty} \left( \frac{f(2^n)}{2^n} \right)^{\frac{d-p-1}{p-1}} = \infty.
\end{equation}
\end{example}

\begin{proof}
Let $B_n = B(o, 2^n)$ be the ball centered at the origin with radius $r_n = 2^n$, and let $A_n = \mathcal{T} \cap B_n$. In $\mathbb{Z}^d$, the volume growth satisfies $\mu(B_n) \asymp 2^{nd}$. By Theorem \ref{wcm}, $\mathcal{T}$ is $p$-thick if and only if
\begin{equation}\label{eq:wiener_zd}
    \sum_{n=1}^{\infty} \left( 2^{n(p-d)} \operatorname{cap}_{p}(A_n) \right)^{\frac{1}{p-1}} = \infty.
\end{equation}
For simplicity of notation, let $q = \frac{d-p-1}{p-1}$. We analyze the criterion by dividing the growth rate of $f$ into two cases.

\noindent \textbf{Case 1:} $\limsup_{n \to \infty} \frac{f(n)}{n} = c > 0$. \\
In this case, the test series \eqref{eq:thorn_series} diverges trivially. We will show that \eqref{eq:wiener_zd} also diverges. 

Along a subsequence $\{n_k\}$ where $f(2^{n_k-1}) \gtrsim c 2^{n_k-1}$, we construct a ball $B_c = B(x_c, R)$ with center $x_c = (3 \cdot 2^{n_k-2}, 0, \dots, 0)$ and radius $R = \min\{c, 1\} 2^{n_k-2}$. 

For any $y = (y_1, y') \in B_c$, we have $\|y\| \le R + 3 \cdot 2^{n_k-2} \le 2^{n_k}$, which means $B_c \subset B_{n_k}$. Moreover, noting that $y_1 \ge 3 \cdot 2^{n_k-2} - R \ge 2^{n_k-1}$ and $f$ is increasing, we have $\|y'\| \le R \le c 2^{n_k-2} < f(2^{n_k-1}) \le f(y_1)$. Hence, $B_c \subset A_{n_k}$. 

Using the capacity lower bound for a ball in $\mathbb{Z}^d$, we obtain
\begin{equation*}
    \operatorname{cap}_p(A_{n_k}) \ge \operatorname{cap}_p(B_c) \gtrsim R^{d-p} \asymp (2^{n_k})^{d-p}.
\end{equation*}
Substituting this into \eqref{eq:wiener_zd}, we have $2^{n_k(p-d)} \operatorname{cap}_p(A_{n_k}) \gtrsim 1$, which implies that the Wiener series diverges.

\noindent \textbf{Case 2:} $\lim_{n \to \infty} \frac{f(n)}{n} = 0$. \\
We estimate $\operatorname{cap}_p(A_n)$ by using two cylinders $C_n$ and $C'_n$.

For the upper bound, $A_n$ is trivially contained in $C_n = \{(x_1, x') : 0 \le x_1 \le 2^n, \|x'\| \le f(2^n)\}$. By Lemma \ref{cap_c}, we have
\begin{equation*}
    \operatorname{cap}_p(A_n) \le \operatorname{cap}_p(C_n) \lesssim 2^n f(2^n)^{d-p-1}.
\end{equation*}

For the lower bound, let $C'_n = \{(x_1, x') : 2^{n-1} \le x_1 \le \frac{3}{4} 2^n, \|x'\| \le f(2^{n-1})\}$. The condition $\lim_{n \to \infty} f(n)/n = 0$ ensures that for large enough $n$, $f(2^{n-1}) \le \frac{1}{8} 2^n$. Thus, for any point in $C'_n$, its distance to the origin is bounded by $x_1 + \|x'\|_1 \le \frac{3}{4} 2^n + \frac{1}{8} 2^n < 2^n$. Hence, $C'_n \subset A_n$. Applying Lemma \ref{cap_c} to $C'_n$ yields
\begin{equation*}
    \operatorname{cap}_p(A_n) \ge \operatorname{cap}_p(C'_n) \gtrsim 2^n f(2^{n-1})^{d-p-1}.
\end{equation*}

Substituting these bounds into \eqref{eq:wiener_zd} produces two series. The upper bound produces $\sum_{n=1}^{\infty} \left( \frac{f(2^n)}{2^n} \right)^q$. The lower bound produces
\begin{equation*}
    \sum_{n=1}^{\infty} \left( \frac{2^n f(2^{n-1})^{d-p-1}}{2^{n(d-p)}} \right)^{\frac{1}{p-1}} = \left(\frac{1}{2}\right)^q \sum_{n=1}^{\infty} \left( \frac{f(2^{n-1})}{2^{n-1}} \right)^q.
\end{equation*}
Therefore, we find that \eqref{eq:wiener_zd} diverges if and only if the test series \eqref{eq:thorn_series} diverges.
\end{proof}

\begin{example}\label{ex:coordinate_axes}
    Assume $d > p>1$ and let $A = \mathbb{Z} \times \{0\}^{d-1}$ denote a one-dimensional coordinate axis in $\mathbb{Z}^d$. The massiveness of its complement $\mathbb{Z}^d \setminus A$ depends strictly on the dimension $d$. 
    If $d = p + 1$, then $\mathbb{Z}^d \setminus A$ is not $p$-massive. 
    However, if $d > p + 1$, then $\mathbb{Z}^d \setminus A$ is $p$-massive but not $D_p$-massive.
\end{example}

\begin{proof}
    Let $B_n = B(o, 2^n)$ be the ball centered at the origin with radius $r_n = 2^n$, and let $A_n = A \cap B_n$. We use Theorem \ref{wcm} with \eqref{wcm2}; that is, $\mathbb{Z}^d \setminus A$ is $p$-massive if and only if
\begin{equation}\label{weiner_axes}
    \sum_{n=1}^{\infty} \left( 2^{n(p-d)} \operatorname{cap}_{p}(A_n,B_n) \right)^{\frac{1}{p-1}} < \infty.
\end{equation}

    First, for $d=p+1$, we calculate $\operatorname{cap}_p(\{x_0\},B'_{n})$ on $\mathbb{Z}^{p}$, where $x_0\in \mathbb{Z}^{p}$ and $B'_{n}=B(x_0,2^{n})$ on $\mathbb{Z}^{p}$.

    Let $u_n$ be the potential function of $(\{x_0\},B'_{n+1})$. By Lemma \ref{lem_eng} with $t=1$, we have
    \begin{align*}
        \operatorname{cap}'_p(\{x_0\},B'_{n+1}) = D'_p(u_n) = -\mu(x_0)\Delta'_p u_n(x_0) := c_n.
    \end{align*}
    Letting $v_n := c_n^{-\frac{1}{p-1}} u_n$, we have $-\mu(x_0)\Delta_p v_n(x_0)=1$. Then, applying Lemma \ref{suph_upb}, we obtain
    \begin{align*}
        v_n(x_0) \lesssim  \left( \frac{1}{\mu'(B_n)} \sum_{x \in B_n} u(x)^p  \right)^{1/p} + \sum^{n+1}_{i=1}\left(\frac{r_i^p}{\mu'(B_i)}\right)^{\frac{1}{p-1}}.
    \end{align*}

    Using the $(p,p)$-Poincaré inequality:
      \begin{equation*}
       \left(\frac{1}{\mu(B)}\sum_{x\in B} |f(x)-f_B|^p\mu(x) \right)^{\frac{1}{p}} \le C r \left(\frac{1}{\mu(2B)}\sum_{x,y\in 2B}|f(y)-f(x)|^p\mu_{xy}\right)^{\frac{1}{p}},
    \end{equation*}
   we have 
    \begin{align*}
        \left( \frac{1}{\mu(B'_n)} \sum_{x \in B'_n} u(x)^p \mu(x) \right)^{1/p} \lesssim \left( \frac{r_n^p}{\mu(B'_n)} \sum_{x,y \in B'_{n+1}} |\nabla_{xy}u|^p \mu(x) \right)^{1/p} \asymp D'_{p}(v_n)=1,
    \end{align*}
   and $\sum\limits^{n+1}_{i=1}\left(\frac{r_i^p}{\mu'(B_i)}\right)^{\frac{1}{p-1}} \asymp n$, therefore
   \begin{align*}
       c_n^{-\frac{1}{p-1}} u_n(x_0) = v_n(x_0) \lesssim n,
   \end{align*}
   which implies that
   \begin{align*}
       \operatorname{cap}'_p(\{x_0\},B'_{n}) \ge c_n \gtrsim n^{1-p}.
   \end{align*}

   Using the same argument as in \eqref{cl_lo}, we obtain
   \begin{align*}
       \operatorname{cap}_p(A_n,B_{n}) \gtrsim 2^n n^{1-p}.
   \end{align*}
   Substituting this into \eqref{weiner_axes}, we have
  \begin{align*}
      \sum_{n=1}^{\infty} \left( 2^{n(p-d)} \operatorname{cap}_{p}(A_n,B_n) \right)^{\frac{1}{p-1}} \gtrsim \sum_{n=1}^{\infty}n^{-1}.
  \end{align*}
  Hence, $\mathbb{Z}^d \setminus A$ is not $p$-massive for $d=p+1$. 

   For $d>p+1$, $\mathbb{Z}^d \setminus A$ is $p$-massive by Example \ref{thorn_thick}.

   For $D_p$ massiveness, we derive it by example \ref{exp_dp}.
\end{proof}

\begin{example}\label{exp_dp}
    Let $p > 1$ and $M \subset \mathbb{Z}^d$. If $\mu(M^c) = \infty$, then $M$ is not $D_p$-massive.
\end{example}

\begin{proof}
    If $d \le p$, then $\mathbb{Z}^d$ is $p$-parabolic, which immediately implies that $M$ is not $D_p$-massive.
     
    If $d > p$, for any non-$p$-parabolic set $\Omega_1$ such that $\Omega_1\subset M$, let
    \[
    A_n = M^c \cap B_n.
    \]
    Then
    \begin{align*}
    	\operatorname{cap}_{p}(\Omega_1,M)
    	&=\operatorname{C}_{p}(M^c, \Omega_1)
    	\ge \lim_{n \to \infty} \operatorname{C}_{p}(A_n, \Omega_1) \\
    	&=\lim_{n \to \infty} \operatorname{cap}_{p}(A_n,\Omega_1^c)
    	\ge \lim_{n \to \infty} \operatorname{cap}_{p}(A_n).
    \end{align*}
    On the other hand, Lemma \ref{lem_capset} yields
    \begin{align*}
    	\operatorname{cap}_{p}(A_n)
    	\gtrsim \mu(A_n)^{\frac{d-p}{d}},
    \end{align*}
    and therefore
    \begin{align*}
    	\operatorname{cap}_{p}(\Omega_1,M)
    	\ge \lim_{n \to \infty} \operatorname{cap}_{p}(A_n)
    	\gtrsim \lim_{n \to \infty} \mu(A_n)^{\frac{d-p}{d}}
    	= \infty.
    \end{align*}
    Since $\Omega_1$ is not $p$-parabolic, the conclusion follows by applying Lemma \ref{dpmcr}. 

\end{proof}

% \begin{example}
%     Let $V$ be the joint of $\mathbb{Z}^{d_1}$ and $\mathbb{Z}^{d_2}$ with respect to an edge. If one of $d_i\le p+1$, then $V$ is $p$ Liouville. If both $d_i> p+1$, $V$ is not $p$ Liouville but $D_p$ Liouville.
% \end{example}

% \begin{proof}
    
% \end{proof}

\section*{Acknowledgments}
The author expresses sincere gratitude to Prof. Grigor'yan from the University of Bielefeld for his invaluable insights and constructive discussions.

\bibliographystyle{plainurl} 
\bibliography{references}

@article{HK01,
  title={Volume growth and parabolicity},
  author={Holopainen, Ilkka and Koskela, Pekka},
  journal={Proc. Amer. Math. Soc.},
  volume={129},
  number={11},
  pages={3425--3435},
  year={2001},
  doi={10.1090/S0002-9939-01-05954-8},
  url={https://doi.org/10.1090/S0002-9939-01-05954-8}
}

@misc{AFS25,
  title={Characterizations of $p$-parabolicity on graphs},
  author={Adriani, Andrea and Fischer, Florian and Setti, Alberto G.},
  year={2025},
  eprint={2507.13696},
  archivePrefix={arXiv},
  primaryClass={math.FA}
}

@article{gri88,
  title={Existence of nontrivial bounded harmonic functions on a {R}iemannian manifold},
  author={Grigor'yan, Alexander},
  journal={Russian Math. Surveys},
  volume={43},
  number={1},
  pages={239--240},
  year={1988},
  publisher={IOP Publishing}
}

@article{hol94,
  title={Rough isometries and $p$-harmonic functions with finite {D}irichlet integral},
  author={Holopainen, Ilkka},
  journal={Rev. Mat. Iberoam.},
  volume={10},
  number={1},
  pages={143--176},
  year={1994},
  doi={10.4171/RMI/148},
  url={https://doi.org/10.4171/RMI/148}
}

@article{HS97phar,
  title={$p$-harmonic functions on graphs and manifolds},
  author={Holopainen, Ilkka and Soardi, Paolo M.},
  journal={Manuscripta Math.},
  volume={94},
  number={1},
  pages={95--110},
  year={1997},
  doi={10.1007/BF02677841},
  url={https://doi.org/10.1007/BF02677841}
}

@article{yam77,
  title={Parabolic and hyperbolic infinite networks},
  author={Yamasaki, Maretsugu},
  journal={Hiroshima Math. J.},
  volume={7},
  number={1},
  pages={135--146},
  year={1977},
  doi={10.32917/hmj/1206135953},
  url={https://doi.org/10.32917/hmj/1206135953}
}

@article{HK95,
  title={Sobolev meets {P}oincar{\'e}},
  author={Haj{\l}asz, Piotr and Koskela, Pekka},
  journal={C. R. Acad. Sci. Paris S{\'e}r. I Math.},
  volume={320},
  pages={1211--1215},
  year={1995}
}

@article{HS97liou,
  title={A strong {L}iouville theorem for $p$-harmonic functions on graphs},
  author={Holopainen, Ilkka and Soardi, Paolo M.},
  journal={Ann. Acad. Sci. Fenn. Math.},
  volume={22},
  pages={205--226},
  year={1997},
  url={https://afm.journal.fi/article/view/134887}
}

@article{IM60,
  title={Potentials and the random walk},
  author={It{\^o}, Kiyosi and McKean, Henry P.},
  journal={Illinois J. Math.},
  volume={4},
  pages={119--132},
  year={1960},
  doi={10.1215/ijm/1255455738},
  url={https://doi.org/10.1215/ijm/1255455738}
}

@article{Lam63,
  title={Wiener's test and Markov chains},
  author={Lamperti, John},
  journal={J. Math. Anal. Appl.},
  volume={6},
  number={1},
  pages={58--66},
  year={1963},
  doi={10.1016/0022-247X(63)90092-1},
  url={https://doi.org/10.1016/0022-247X(63)90092-1}
}

@article{McK61,
  title={A problem about prime numbers and the random walk},
  author={McKean, Henry P.},
  journal={Illinois J. Math.},
  volume={5},
  pages={351},
  year={1961},
  doi={10.1215/ijm/1255629834},
  url={https://doi.org/10.1215/ijm/1255629834}
}

@article{BC15,
  title={$\alpha$-stable random walk has massive thorns},
  author={Bendikov, Alexander and Cygan, Wojciech},
  journal={Colloq. Math.},
  volume={138},
  number={1},
  pages={105--129},
  year={2015},
  doi={10.4064/cm138-1-7},
  url={https://doi.org/10.4064/cm138-1-7}
}

@article{KW92,
  title={The Dirichlet problem at infinity for random walks on graphs with a strong isoperimetric inequality},
  author={Kaimanovich, Vadim A. and Woess, Wolfgang},
  journal={Probab. Theory Related Fields},
  volume={91},
  number={3-4},
  pages={445--466},
  year={1992},
  doi={10.1007/BF01192066},
  url={https://doi.org/10.1007/BF01192066}
}

@article{Kur13,
  title={The Dirichlet Problem for $p$-Harmonic Functions on a Network},
  author={Kurata, Hisayasu},
  journal={Interdiscip. Inform. Sci.},
  volume={19},
  number={2},
  pages={121--127},
  year={2013},
  doi={10.4036/iis.2013.121},
  url={https://doi.org/10.4036/iis.2013.121}
}

@article{Puls14,
  title={The $p$-harmonic boundary and $D_p$-massive subsets of a graph of bounded degree},
  author={Puls, Michael J.},
  journal={Bull. Aust. Math. Soc.},
  volume={89},
  number={1},
  pages={149--158},
  year={2014},
  doi={10.1017/S0004972713000439},
  url={https://doi.org/10.1017/S0004972713000439}
}

@article{Bj09,
  title={Necessity of a Wiener type condition for boundary regularity of quasiminimizers and nonlinear elliptic equations},
  author={Bj{\"o}rn, Jana},
  journal={Calc. Var. Partial Differential Equations},
  volume={35},
  number={4},
  pages={481--496},
  year={2009},
  doi={10.1007/s00526-008-0216-z},
  url={https://doi.org/10.1007/s00526-008-0216-z}
}

@article{KM92,
  title={Degenerate elliptic equations with measure data and nonlinear potentials},
  author={Kilpel{\"a}inen, Tero and Mal{\'y}, Jan},
  journal={Ann. Sc. Norm. Super. Pisa Cl. Sci. (4)},
  volume={19},
  number={4},
  pages={591--613},
  year={1992},
  url={https://www.numdam.org/item/ASNSP_1992_4_19_4_591_0/}
}

@book{Bar17,
  title={Random Walks and Heat Kernels on Graphs},
  author={Barlow, Martin T.},
  volume={438},
  series={London Mathematical Society Lecture Note Series},
  year={2017},
  publisher={Cambridge University Press}
}

@article{KM94,
  title={The Wiener test and potential estimates for quasilinear elliptic equations},
  author={Kilpel{\"a}inen, Tero and Mal{\'y}, Jan},
  journal={Acta Math.},
  volume={172},
  number={1},
  pages={137--161},
  year={1994},
  doi={10.1007/BF02392793},
  url={https://doi.org/10.1007/BF02392793}
}

@article{CS93,
  title={Isop{\'e}rim{\'e}trie pour les groupes et les vari{\'e}t{\'e}s},
  author={Coulhon, Thierry and Saloff-Coste, Laurent},
  journal={Rev. Mat. Iberoam.},
  volume={9},
  number={2},
  pages={293--314},
  year={1993},
  doi={10.4171/RMI/138},
  url={https://doi.org/10.4171/RMI/138}
}

@article{CS95,
  title={Vari{\'e}t{\'e}s riemanniennes isom{\'e}triques {\`a} l'infini},
  author={Coulhon, Thierry and Saloff-Coste, Laurent},
  journal={Rev. Mat. Iberoam.},
  volume={11},
  number={3},
  pages={687--726},
  year={1995},
  doi={10.4171/RMI/190},
  url={https://doi.org/10.4171/RMI/190}
}

@article{Sal97,
  title={Inequalities for $p$-superharmonic functions on networks},
  author={Saloff-Coste, Laurent},
  journal={Rend. Sem. Mat. Fis. Milano},
  volume={65},
  number={1},
  pages={139--158},
  year={1997},
  publisher={Springer}
}

@article{CK04,
  title={Geometric Interpretations of {$L^p$}-Poincar{\'e} Inequalities on Graphs with Polynomial Volume Growth},
  author={Coulhon, Thierry and Koskela, Pekka},
  journal={Milan J. Math.},
  volume={72},
  pages={209--248},
  year={2004},
  publisher={Birkh{\"a}user Basel},
  doi={10.1007/s00032-004-0027-4},
  url={https://doi.org/10.1007/s00032-004-0027-4}
}

@article{BCLS95,
  title={Sobolev inequalities in disguise},
  author={Bakry, Dominique and Coulhon, Thierry and Ledoux, Michel and Saloff-Coste, Laurent},
  journal={Indiana Univ. Math. J.},
  volume={44},
  number={4},
  pages={1033--1074},
  year={1995},
  doi={10.1512/iumj.1995.44.2019},
  url={https://doi.org/10.1512/iumj.1995.44.2019}
}
\end{document}